\documentclass[11pt,a4paper]{article}
\usepackage{amsmath,amssymb,amsfonts,amsthm}
\usepackage[pdftex]{graphicx, color}

\newcommand{\E}{{\bf{E}}}
\newcommand{\PP}{{\bf{P}}}
\newcommand{\Var}{{\bf{Var}}}

\newtheorem{tm}{Theorem}

\newtheorem{lem}{Lemma}

\long\def\symbolfootnote[#1]#2
{\begingroup%
\def\thefootnote{\fnsymbol{footnote}}\footnote[#1]{#2}\endgroup}

\begin{document}

\bibliographystyle{plain}
\parindent=0pt
\centerline{\LARGE \bfseries Random intersection graph process}

\par\vskip 3.5em

\centerline{
Mindaugas Bloznelis\symbolfootnote[1]{Faculty of Mathematics and Informatics, 
Vilnius University, 03225 Vilnius, Lithuania}
\
\
and
\
\
Micha\l   \ Karo\'nski\symbolfootnote[2]{Faculty of Mathematics and Computer science,
Adam Mickiewicz University, 60769 Pozna\'n, Poland}
}
%\centerline{Mindaugas  Bloznelis}

\vglue1truecm

%\centerline{Vilnius University, Faculty of mathematics and informatics,} 
%\centerline{Naugarduko 24, Vilnius,
%03225, Lithuania}

%\footnotetext[*]{Faculty of Mathematics and Informatics, Vilnius University, 03225 Vilnius, Lithuania}
%\centerline{ Mindaugas Bloznelis\footnote{Research supported by the Research Council of Lithuania Grant MIP-053}}

\bigskip

% Faculty of Mathematics and Informatics, Vilnius University,

% Naugarduko 24, LT-03225 Vilnius, Lithuania

%E-mail:   mindaugas.bloznelis$@$mif.vu.lt

%http://www.mif.vu.lt/$\sim$bloznelis
%\par\vskip 3.5em

\begin{abstract} We introduce a random intersection graph process aimed 
at modeling sparse evolving affiliation networks that admit tunable  (power 
law) degree distribution and  assortativity and clustering coefficients. 
We show the asymptotic degree distribution and provide explicit asymptotic formulas for  
assortativity and clustering coefficients.

\end{abstract}

\smallskip
{\bfseries keywords:}   random graph process, random intersection graph, 
degree distribution, power law, clustering, assortativity
\vskip 1em


\vskip 2.5em

\section{Introduction}

Given non-negative weights $x=\{x_i\}_{i\ge 1}$ and $y=\{y_j\}_{j\ge 1}$, 
and a nondecreasing positive sequence
 $\{\tau(t)\}_{t\ge 1}$, satisfying $\lim_{t\to+\infty}\tau(t)=+\infty$,   
let $H_{x,y}$ be  the   random bipartite graph with bipartition $V=\{v_1,v_2,\dots\}$ 
and $W=\{w_1,w_2,\dots\}$, where edges $\{w_i,v_j\}$ are inserted independently and 
with probabilities 
\begin{equation}\label{ppppp}
p_{ij}=\min\Bigl\{1,\frac{x_iy_j}{\sqrt{ij}}\Bigr\}
{\mathbb I}_{\{ a\tau(j)\le i\le b \tau(j)\}}.
\end{equation}
Here $b>a>0$ are fixed numbers.
 $H_{x,y}$ defines the random intersection graph $G_{x,y}$ on the vertex set $V$ such 
that any
$u, v\in V$ are declared adjacent (denoted $u\sim v$) whenever they have a common 
neighbor in $H_{x,y}$.
% (i.e., for some $w\in W$ both edges $\{w, u\}$ and $\{w, v\}$ are present in $H_{x,y}$).

Consider, for example, a library 
%with items $w_1, w_2, \dots$, 
where a new item $w_i$ is 
acquired at time $i$, and 
where a new user $v_j$  is registered at time j. User $v_j$ picks at random
items 
%$w_s$ 
from a "contemporary literature collection" 
$\{w_i: a\tau(j)\le i\le b\tau(j)\}$ relevant to time $j$ (the interval $\{i:\, a\tau(j)\le i\le b\tau(j)\}$ 
can also be considered as the lifetime of the user $v_j$).
Every actor $v_j$ and  every
item $w_i$ is assigned weight $y_j$ and $x_i$ respectively. 
These weights model the activity of actors and attractiveness of literature items. 
Now, assume that up to time $t$ the library has acquired items 
$\{w_1,\dots, w_{\tau_*(t)}\}=:W_{\tau_*(t)}$, where 
$\tau_*:{\mathbb N}\to{\mathbb N}$ is a given nondecreasing function satisfying 
 $\lim_{t\to+\infty}\tau_*(t)=+\infty$.
The subgraph $H_{x,y}(t)$ of $H_{x,y}$ induced by the bipartition  
$V_t=\{v_1,v_2,\dots, v_t\}$ and
$W_{\tau_*(t)}$
defines the random intersection graph $G_{x,y}(t)$ on the vertex set 
$V_t$:  vertices $u,v\in V_t$ are declared adjacent whenever they have a common neighbor 
in $H_{x,y}(t)$. 
%Clearly, $G_{x,y}(t)$ is a subgraph of $G_{x,y}$.
The graph $H_{x,y}(t)$  represents a snapshot taken at time $t$ of the ``library" records, 
while the graph $H_{x,y}$  
%keeps
%represents
%contains
 %archives 
% records 
 shows the complete history of the ``library". Graphs  $G_{x,y}(t)$  and $G_{x,y}$ 
represent 
%respective 
adjacency relations
(between users) observed up to time $t$ and during the whole lifetime of the ``library'', 
respectively. 
%We 
%call the sequence $\{G_{x,y}(t)\}_{t\ge 1}$ the random intersection graph process. 
Assuming, in addition, that $x$ and $y$ are realized 
values of iid sequences $X=\{X_i\}_{i\ge 1}$ and $Y=\{Y_j\}_{j\ge 1}$ we obtain 
the random graph $G_{X,Y}$  and the random 
graph process
$\{G_{X,Y}(t)\}_{t\ge 1}$. The parameters of such a network model are the probability distributions of  $X_1,Y_1$, the functions 
$\tau, \tau^*$ and the cutt-offs $a<b$.

Random intersection graph $G_{X,Y}$  is aimed at modeling sparse
evolving affiliation networks that admit 
a power law degree distribution and non-vanishing  
{\it clustering} and assortativity coefficients. 
We first observe that choosing inhomogeneous weight sequences $x$ and $y$ one 
typically obtains an inhomogeneous degree sequence of the graph $G_{x,y}$: vertices 
with larger weights attract larger numbers of neighbours. Consequently, 
in the case where the probability distributions of $X_1$ and $Y_1$ have heavy tails, 
we may expect to obtain a heavy tailed (asymptotic) degree distribution in the random 
graph $G_{X,Y}$.
Secondly, we observe that if the set $W(t)$ of items selected by a user $v_t$ 
is (stochastically) bounded and the lifetimes of two neighbours of $v_t$, 
say $v_s$ and $v_u$, intersect, then with a non-vanishing probability 
$v_s$ and $v_u$ share an item from $W(t)$. 
Consequently, the conditional probability
$\alpha_{t|su}=\PP(v_s\sim v_u|v_s\sim v_t, v_t\sim v_u)$, called the clustering coefficient, is positive and bounded away from zero. 
In particular, the underlying bipartite graph structure serves  as a
clustering mechanism.

Let us compare our model with the model of evolving network considered recently by 
Britton, Lindholm and 
Turova (2011) \cite{BrittonLT2011}, 
(see also \cite{BrittonL2010}, \cite{Turova2002} \cite{Turova2007}). 
In their model vertices are prescribed weights, called social indices, 
and
%, given these weights,  
a vertex  $v_t$ with social index $s_t$ creates new edges at a 
rate proportional to $s_t$.  Clearly, both weight sequences  
$\{y_t\}_{t\ge 1}$ and $\{s_t\}_{t\ge 1}$ have the same purpose of 
modeling inhomogeneity of adjacency relations 
(hence both models possess a power law asymptotic degree distribution). But the model of  
Britton, Lindholm and Turova (2011)\cite{BrittonLT2011}  does not have
 the clustering property.  We remark, that the role of a  
bipartite structure in understanding/explaining clustering properties of some
social  networks has been discussed in
 Newman, Watts, and  Strogatz (2002)
\cite{Newman+W+S2002}. Furthermore,
empirically observed clustering properties of real affiliation networks
have been reproduced with remarkable accuracy by related models of random intersection graphs,
see  
\cite{Bloznelis2011+}, \cite{BloznelisKurauskas2012}.

In the present paper we only consider the graph $G_{X,Y}$. 
We show  the  asymptotic distribution of
the degree  $d(v_t)$ of a vertex $v_t$ as  time $t\to+\infty$. We also obtain 
explicit asymptotic expressions for clustering coefficients 
$\alpha_{t|s,u}$, $\alpha_{s|t,u}$, $\alpha_{u|s,t}$, for $s,t,u\to+\infty$  
such that $s<t<u$, and for 
the assortativity coefficient (Pearson's correlation coefficient 
between degrees of adjacent vertices)
\begin{equation}\label{r-st-0}
r_{s,t}
=
\frac
{\E_{st}d(v_s)d(v_t)-\E_{st} d(v_s)\E_{st} d_v(t)}
{\sqrt{\Var_{st}d(v_s)\Var_{st}d(v_t)}}.
\end{equation}
Here $\E_{st}$ denotes the conditional expectation given the event $v_s\sim v_t$ and
$\Var_{st}d(v_s)=\E_{st}d^2(v_s)-(\E_{st}d_v(s))^2$. We remark that (empirical) clustering 
and assortativity coefficients are  commonly used characteristics 
of statistical dependence
of adjacency relations of real networks.

Our results are stated in Section 2. Proofs are given in Section 3.

 \section{Results}

{\bf Degree.}
We first present our results on the asymptotic degree
distribution 
%of the degree $d(v_t)$ of a vertex $v_t$ 
in $G_{X,Y}$. 
We obtain a compound probability distribution
in the case where $\tau(t)$ grows linearly in $t$ (clustering regime). 
For $\tau(t)$ growing faster than linearly in $t$, we obtain a mixed Poisson asymptotic degree distribution.
We denote 
$a_k=\E X_1^k$, and $b_k=\E Y_1^k$.

\begin{tm}\label{T1} Let $b>a>0$. Let $\tau(t)=t$. 
 Suppose  that $\E X_1^2<\infty$ and $\E Y_1<\infty$. For $t\to+\infty$ 
the random variable  $d(v_t)$ converges in distribution to the random variable
\begin{equation}\label{d*1}
d_*=\sum_{j=1}^{\Lambda_1}\varkappa_j, 
\end{equation}
where $\varkappa_1,\varkappa_2,\dots$ are independent 
and identically distributed random variables independent of the random variable $\Lambda_1$.
They are distributed as follows. For $r=0,1,2,\dots$, we have
\begin{equation}\label{d*1++}
\PP(\varkappa_1=r)=\frac{r+1}{\E\Lambda_2}\PP(\Lambda_2=r+1)
\qquad
{\text{and}}
\qquad
\PP(\Lambda_i=r)=\E \,e^{-\lambda_i}\frac{\lambda_i^r}{r!},
\qquad
i=1,2.
\end{equation}
Here $\lambda_1=2(b^{1/2}-a^{1/2})a_1Y_1$ and $\lambda_2=2(a^{-1/2}-b^{-1/2})b_1X_1$.
\end{tm}

The second moment condition $\E X_1^2<\infty$ of Theorem \ref{T1} seems to be 
redundant.  
% and could perhaps be waived by involving an additional truncation argument 
%in the proof. 

\begin{tm}\label{T2} Let $b>a>0$ and $\nu>1$. Let $\tau(t)=t^{\nu}$, $t=1,2\dots$. 
 Suppose  that $\E X_1^2<\infty$ and $\E Y_1<\infty$. For $t\to+\infty$ 
the random variable  $d(v_t)$ converges in distribution to the random variable 
$\Lambda_3$ having the probability distribution
\begin{equation}\label{d*2}
\PP(\Lambda_3=r)=\E \,e^{-\lambda_3}\frac{\lambda_3^r}{r!},
\qquad
r=0,1,2,\dots.
\end{equation}
Here $\lambda_3=\gamma a_2b_1Y_1$ and
$\gamma=4\nu(b^{1/2\nu}-a^{1/2\nu})(a^{-1/2\nu}-b^{-1/2\nu})$.
\end{tm}

{\it Remark 1.} The result of Theorem \ref{T2} extends to a more general class
of increasing 
nonnegative  
functions $\tau$. In particular, assuming that  
\begin{equation}\label{condition-h}
 \lim_{t\to+\infty}\frac{t}{\tau(t)}=0,
\qquad
\sup_{t>1}\frac{\tau^{-1}(2t)}{\tau^{-1}(t)}<\infty, 
\end{equation}
and that
there exists finite limit
\begin{displaymath}
\gamma^*
=
\lim_{t\to+\infty}
t^{-1/2}
\sum_{a\tau(t)\le i\le b\tau(i)}i^{-1}
\sum_{j:\, a\tau(j)\le i\le b\tau(j)}j^{-1/2},
\end{displaymath}
we obtain the convergence in distribution of $d(v_t)$ to $\Lambda_3$ defined by (\ref{d*2})
with $\lambda_3=\gamma^*a_2b_1Y_1$. Here  $\tau^{-1}$ denotes the inverse of $\tau$
(i.e., $\tau(\tau^{-1}(t))=t$).

{\it Remark 2.} The function $\tau(t)=t\ln t$, which grows slower than any power $t^{\nu}$, $\nu>1$, satisfies conditions of Remark 1 with 
$\gamma^*=4(a^{-1/2}-b^{-1/2})(b^{1/2}-a^{1/2})$.
Furthermore, the functions $\tau_1(t)=e^{\ln^2t}$ and $\tau_2(t)=e^{t}$, that grow faster than any power $t^{\nu}$, satisfy conditions of Remark 1 with 
$\gamma^*=0$.

\medskip
{\bf Clustering.} Our next result, Theorem \ref{T3}, provides explicit asymptotic formulas for clustering coefficients. We note that for $s<t<u$ the conditional
probabilities $\alpha_{s|tu}$,
 $\alpha_{t|su}$ and  $\alpha_{u|st}$ are all different and, given $0<a<b$,  mainly depend on the ratios $s/t$, $s/u$ and $t/u$.
 Denote $p_{\Delta}:=
p_{\Delta}(s,t,u)=
\PP(v_s\sim v_t, v_s\sim v_u, v_t\sim v_u)$ the probability that $v_s,v_t,v_u$ make up a triangle.

\begin{tm}\label{T3} Let $b>a>0$. Let $\tau(t)=t$. 
 Suppose  that $\E X_1^3<\infty$ and $\E Y_1^2<\infty$.  Assume that $s,t,u\to+\infty$ 
so that
 $s<t<u$ and  $\lceil au\rceil\le \lfloor bs\rfloor$.  
We have
\begin{eqnarray}\label{p(Delta)}
p_{\Delta}
&
= 
&
\frac{a_3b_1^3}{\sqrt{stu}}
\left(\frac{2}{\sqrt{au}}-\frac{2}{\sqrt{bs}}\right)+o(t^{-2}),
\\
\label{alpha-t}
\alpha_{t|su}
&
=
&
\frac{p_{\Delta}}{p_{\Delta}+a_2^2b_1^2b_2t^{-1}(su)^{-1/2}\delta_{t|su}}+o(1), 
\\
\label{alpha-s}
\alpha_{s|tu}
&
=
&
\frac{p_{\Delta}}{p_{\Delta}+a_2^2b_1^2b_2s^{-1}(tu)^{-1/2}\delta_{s|tu}}+o(1), 
\\
\label{alpha-u}
\alpha_{u|st}
&
=
&
\frac{p_{\Delta}}{p_{\Delta}+a_2^2b_1^2b_2u^{-1}(st)^{-1/2}\delta_{u|st}}+o(1).
\end{eqnarray}
Here
\begin{eqnarray}\nonumber
\delta_{t|su}
&
=
&
\ln (u/t)\ln(t/s)+ \ln (u/t)\ln(bs/au)+\ln (t/s)\ln(bs/au)+\ln^2(bs/au),
\\
\nonumber
\delta_{s|tu}
&
=
&\ln (u/t)\ln(bs/au)+\ln^2(bs/au),
\\
\nonumber
\delta_{u|st}
&
=
&
\ln (t/s)\ln(bs/au)+\ln^2(bs/au).
\end{eqnarray}
\end{tm}

We remark that the condition  $\lceil au\rceil\le \lfloor bs\rfloor$ of Theorem \ref{T3}
excludes the trivial case where $p_{\Delta}\equiv 0$. Indeed, for $s<u$, the 
converse inequality 
$\lceil au\rceil> \lfloor bs\rfloor$ means that the lifetimes of $v_s$ and $v_u$ do not
intersect and, therefore, we have  $\PP(v_s\sim v_u)\equiv 0$.
In addition, the inequality $\lceil au\rceil\le \lfloor bs\rfloor$ implies that positive numbers
$\delta_{t|su}$, $\delta_{s|tu}$, $\delta_{u|st}$ are bounded from above  by a constant 
(only depending on $a$ and $b$).

 {\bf Assortativity.} Let us now consider the sequence of random variables $\{d(v_t)\}_{t\ge 1}$. 
We assume that $\tau(t)=t$. From Theorem 
\ref{T1} we know about the possible limiting distributions for $d(v_t)$. 
Moreover, from the fact
that $G_{X,Y}$ is sparse we can conclude that, for any given $k$, the random variables 
$d(v_t), d(v_{t+1}),\dots, d(v_{t+k})$ are asymptotically independent as $t\to+\infty$.
An interesting question is about the statistical dependence between $d(v_s)$ and $d(v_t)$ 
if we know,
in addition, that vertices $v_s$ and $v_t$ are adjacent in $G_{XY}$. 
We assume  that $s<t$ and  let $s,t\to+\infty$ so that  $bs-at\to+\infty$. 
Note  that the 
latter condition ensures that the shared lifetime of $v_s$ and $v_t$ tends to infinity as 
$s,t\to+\infty$. In this case we obtain that conditional moments
\begin{eqnarray}\label{s-momentai}
&&
 \E_{st}d(v_s)= \E_{st}d(v_t)+o(1)=\delta_1+o(1),
\\
\nonumber
&&
\E_{st}d^2(v_s)=\E_{st}d^2(v_t)+o(1)=\delta_2+o(1),
\\
\nonumber
&&
\E_{st}d(v_s)d(v_t)= \delta_2-\Delta+o(1),
%
%
%
%1+h_1^{-1}(3h_2+4h_3+h_4+4h_5+4h_7)+o(1)
\end{eqnarray}
are asymptotically constant. Here $\Delta =h_1^{-1}(2h_3+2h_5+4(h_6-h_7))$ and 
\begin{displaymath}
\delta_1=1+h_1^{-1}(h_2+2h_3),
\qquad
 \delta_2=1+h_1^{-1}(3h_2+6h_3+h_4+6h_5+4h_6).
\end{displaymath}
Furthermore, we denote
\begin{eqnarray}
\label{estLs}
&&
 h_1=a_2b_1^2,
\qquad
h_2=a_3b_1^3{\tilde \gamma},
\qquad
h_3=a_2^2b_1^2b_2{\tilde \gamma}(\sqrt{b}-\sqrt{a}),
\\
\nonumber
&&
h_4=a_4b_1^4{\tilde \gamma}^2,
\qquad
h_5=a_2a_3b_1^3b_2{\tilde \gamma}^2(\sqrt{b}-\sqrt{a}),
\\
\nonumber
&&
h_6=a_2^3b_1^3b_3{\tilde \gamma}^2(\sqrt{b}-\sqrt{a})^2,
\qquad
h_7=a_2^3b_1^2b_2^2{\tilde \gamma}^2(\sqrt{b}-\sqrt{a})^2,
\end{eqnarray}
and ${\tilde \gamma}=2(a^{-1/2}-b^{-1/2})$. A sketch of the derivation of relations  
(\ref{s-momentai})
is given in Section 3 below.

Finally, from  (\ref{r-st-0}) and (\ref{s-momentai}) we obtain that the 
assortativity coefficient 
\begin{eqnarray}\label{r-st}
 r_{st}=1-\frac{\Delta}{\delta_2-\delta_1^2}+o(1)
\end{eqnarray}
is asymptotically constant.

\bigskip

We note that each vertex of the graph $G_{X,Y}$ can be identified with the random subset
of $W$, consisting of items selected by that vertex, and two vertices are adjacent 
in $G_{X,Y}$
whenever their subsets intersect. 
%In the literature, 
Graphs 
describing such adjacency relations between members
of a {\it finite} family 
 ${\tilde V}=\{{\tilde v}_1,\dots, {\tilde v}_n\}$ 
of random subsets of a given {\it finite} set 
${\tilde W}=\{w_1,\dots, w_m\}$ are called random intersection graphs, see
%. Up to our best knowledge
%random intersection graphs were introduced in 
\cite{karonski1999}, \cite{Singer1995} 
and \cite{godehardt2001}. Our  graph $G_{X,Y}$ is, therefore, a random intersection 
graph
evolving in time.
One important application of random intersection graphs, 
defined by random subsets of
fixed size, is the  model of  a secure wireless sensor network that uses 
random predistribution
of keys introduced in  \cite{eschenauer2002}. Another potential application
of random intersection graphs is  the statistical analysis and modeling
of affiliation networks. For example, they are useful in explaining clustering properties of
 the actor network, see
\cite{Bloznelis2011+}, \cite{BloznelisKurauskas2012}. 
Finally, we mention 
%a related work on 
that asymptotic degree distribution and clustering properties 
of random intersection graphs have been studied  in
\cite{Bloznelis2008},
\cite{Bloznelis2011+},
\cite{Deijfen},
\cite{GJR2010},
\cite{JKS},
\cite{JaworskiStark2008},
\cite{stark2004}.
%he influence of clustering on the epidemic spread  was discussed in 
%\cite{Britton2008}.

{\bf Concluding remarks.}
We have shown that the random graph  $G_{X,Y}$ admits tunable asymptotic degree 
distribution (icluding the power law) and  clustering and assortativity 
coefficients. 
An interesting problem were to study 
%random graph processes 
$G_{X,Y}$ and $\{G_{X,Y}(t)\}_{t\ge 1}$ 
in the case where  deterministic cutt-offs $a<b$ in (\ref{ppppp}) 
are replaced by random cutt-offs 
$A_j\le B_j$ (so that the lifetime $[A_j\tau(j),B_j\tau(j)]$  of an actor $v_j$ were random). Furthermore,  abrupt cutt-offs can be
replaced by some smooth cutt-off functions.

\section{Proofs}
We first introduce some notation.
Then we prove Theorems \ref{T1}, \ref{T2}, \ref{T3}. The proof of Remark 1 goes along the lines of the proof of Theorem \ref{T2} and is omitted.

Throughout the proof limits are taken as $t\to+\infty$, if not stated otherwise. 
By $c$ we denote  positive  numbers 
which may only depend on $a,b$ and $\tau$. We remark that $c$ may attain different values 
in different places.
 We say that a sequence of random variables $\{\zeta_t\}_{t\ge 1}$ converges to zero 
in probability (denoted $\zeta_t=o_P(1)$) whenever $\limsup_t\PP(|\zeta_t|>\varepsilon)=0$
for each $\varepsilon>0$. The sequence $\{\zeta_t\}_{t\ge 1}$ is called stochastically bounded
(denoted $\zeta_t=O_P(1)$) whenever for each $\delta>0$ there exists $N_{\delta}>0$ such that 
$\limsup_t\PP(|\zeta_t|>N_{\delta})<\delta$.

Time intervals  
\begin{equation}\label{TT}
T_t=\{i: a\tau(t)\le i\le b\tau(t)\},
\qquad
 T^*_i=\{j: a\tau(j)\le i\le b\tau(j)\}
 \end{equation}
 can be interpreted as  lifetimes of the actor $v_t$ and attribute $w_i$ respectively. 
Here and below elements of $V$ are called actors, elements of $W$ are called attributes.
The oldest and youngest actors that may  establish a communication link with $v_t$ 
are denoted $v_{t_{-}}$ and  $v_{t_{+}}$. Here 
\begin{displaymath}
t_{-}=\min\{j: T_j\cap T_t\not=\emptyset\},
\qquad
t_{+}=\max\{j: T_j\cap T_t\not=\emptyset\}.
\end{displaymath}

The event "edge $\{w_i, v_j\}$ is present in $H_{X,Y}$" is denoted $w_i\to v_j$. 
Introduce random variables
\begin{eqnarray}\nonumber
&&
 {\mathbb I}_{ij}={\mathbb I}_{\{w_i\to v_j\}},
\qquad
{\mathbb I}_i={\mathbb I}_{it},
\qquad
u_i=\sum_{j\in T^*_i\setminus\{t\}}{\mathbb I}_{ij},
\qquad
L=L_t=\sum_{i\in T_t}u_i{\mathbb I}_i,
\\
\label{b-sum}
&&
b_{k}(I)=\sum_{j\in I}Y_{j}^kj^{-k/2},
\qquad
a_{k}(I)=\sum_{i\in I}X_i^ki^{-k/2},
\qquad
I\subset {\mathbb N},
\\
\label{QXY}
&&
\lambda_{ij}=X_iY_j/\sqrt{ij},
\qquad
Q_{XY}(t)=\sum_{i\in T_t}\lambda_{it}\sum_{j\in T_i^*\setminus\{t\}}\lambda_{ij}\min\{1,\lambda_{ij}\}.
\end{eqnarray}
We remark, that $u_i$ counts all neighbours of $w_i$ in $H_{X,Y}$ 
belonging to the set $V\setminus \{v_t\}$, and $L_t$ counts all paths of length $2$ in $H_{X,Y}$ starting from $v_t$.
Introduce  events 
\begin{displaymath}
{\cal A}_t=\{\lambda_{it}\le 1, \ i\in T_t\},
\qquad
{\cal B}_t(\varepsilon)=\{Y_j\le \varepsilon^2j,\ j\in [t_{-},t_{+}]\setminus\{t\}\}, 
\quad 
\varepsilon>0.
\end{displaymath}
By ${\tilde \PP}$ and ${\tilde \E}$ we denote the conditional probability and 
expectation given $X, Y$. 
The conditional probability and expectation given $Y$ is denoted $\PP_X$ and $\E_X$.
By $\PP_t$ and $\E_t$ we denote the conditional probability and expectation given $Y_t$.
By  $d_{TV}(\zeta,\xi)$ we denote the total variation 
distance between the probability distributions
of
 random variables $\zeta$ and $\xi$. In the case where $\zeta,\xi$ and $X, Y$ are 
defined on the same probability space, we denote 
 by 
 ${\tilde d}_{TV}(\zeta,\xi)$ the total 
variation distance between the conditional distributions of $\zeta$ and $\xi$ given $X,Y$.

In the proof  we use the following 
simple fact. For a uniformly bounded 
sequence of random variables $\{\zeta_t\}_{t\ge 1}$ 
(i.e., $\exists$ nonrandom $h>0$ such that 
$\forall t$\, $0<\zeta_t<h$ almost surely)  we have
\begin{equation}\label{fact}
 \zeta_t=o_P(1)
\quad
\Rightarrow
\quad
\E\zeta_t=o(1).
\end{equation}
In particular, given
a sequence of bivariate random vectors $\{(\phi_t, \psi_t)\}_{t\ge 1}$, 
defined on the same probability space as $X,Y$, we have
\begin{equation}\label{fact*}
 {\tilde d}_{TV}(\phi_t, \psi_t)=o_P(1)
\quad
\Rightarrow
\quad
d_{TV}(\phi_t,\psi_t)=o(1).
\end{equation}

\subsection{Proof of Theorem \ref{T1}}

Before the proof we collect auxiliary results.
For $\tau(t):=t$ and $T_t$, $T_i^*$ defined in (\ref{TT}), we have 
\begin{eqnarray}\label{Sumos}
&&
\sum_{i\in T_t}i^{-1/2}=t^{1/2}\gamma_1+rt^{-1/2}, 
\qquad
\sum_{j\in T_i^*}j^{-1/2}=i^{1/2}\gamma_2+r'i^{-1/2},
\\
\nonumber
&&
\gamma_1:=2(b^{1/2}-a^{1/2}),
\qquad
\qquad
\gamma_2:=2(a^{-1/2}-b^{-1/2}),
\end{eqnarray}
where $|r|, |r'|\le c$.

%+++++++++++++++++++++++++++++++++++++++++++++++++++++

\begin{lem}\label{d-L} Let $t\to+\infty$. Assume that $\E X_1^2<\infty$ and $\E Y_1<\infty$.
We have 
\begin{eqnarray}
\label{d-bbb}
&&
\forall \varepsilon>0
\qquad
\PP({\cal B}_t(\varepsilon))=1-o(1),
\\
\label{d-SSS}
&&
t^{-1}b_2([t_{-}, t_{+}]\setminus\{t\})=o_P(1),
\\
\label{d-L21}
&&
\PP(d(v_t)\not=L_t)=o(1),
\\
\label{d-L22}
&&
\PP({\cal A}_t)=1-o(1),
\\
\label{d-L23}
&&
Q_{XY}(t)=o_P(1),
\qquad
\E Q_{XY}(t)=o(1).
\end{eqnarray}
%Denote $\gamma=2(a^{-1/2}-b^{-1/2})$. 
For any integers $t>a^{-1}(b+b^{-1})$ and $i\in T_t$, and any $0< \varepsilon<1$ we have 
\begin{eqnarray}
\label{eb1}
&&
|\E a_1(T_t)-a_1\gamma_1t^{1/2}|\le ca_1t^{-1/2},
\qquad
|\E b_1(T_i^*\setminus\{t\})-b_1\gamma_2i^{1/2}|\le cb_1i^{-1/2},
\\
\label{eb2}
&&
\E|b_1(T_i^*\setminus\{t\})-b_1\gamma_2 i^{1/2}|{\mathbb I}_{{\cal B}_t(\varepsilon)}
\le 
ci^{1/2}(\varepsilon b_1^{1/2} +\E Y_1{\mathbb I}_{\{Y_1>\varepsilon^2t_{-}\}})+cb_1i^{-1/2},
\\
\label{eb2+A}
&&
\E|a_1(T_t)-a_1\gamma_1t^{1/2}|\le c a_2^{1/2}.
\end{eqnarray}
\end{lem}

\begin{proof}[Proof of Lemma \ref{d-L}]
Proof of (\ref{d-bbb}).
We estimate the probability of the complement event ${\overline {\cal B}}_t(\varepsilon)$ 
using the union bound and Markov's inequality
\begin{displaymath}
\PP({\overline {\cal B}}_t(\varepsilon))
\le 
\sum_{t_{-}\le j\le t_{+}}\PP(Y_i>\varepsilon^2 j)
=
t_{+}\PP(Y_1>\varepsilon^2 t_{-})
\le
\varepsilon^{-2}(t_{+}/t_{-})\E Y_1{\mathbb I}_{\{Y_1>\varepsilon^2 t_{-}\}}=o(1).
\end{displaymath} 
Here we estimate $t_{+}/t_{-}\le c$ and invoke the bound $\E Y_1{\mathbb I}_{\{Y_1>s\}}=o(1)$,
for $s\to+\infty$.

Proof of  (\ref{d-SSS}). 
Denote ${\hat b}_2(t)=t^{-2}\sum_{1\le j\le t}Y_j^2$. We note that $\E Y_1<\infty$ implies
${\hat b}_2(t)=o_P(1)$. 
The latter bound in combination with the simple inequality $t_{+}/t_{-}\le c$
implies (\ref{d-SSS}).
%
%
%
%Denote $b^*_k(t)=\sum_{j\in [t_{-}, t_{+}]\setminus \{t\}}Y_j^k$. 
%The inequality
%\begin{displaymath}
%(at-1)/b<t_{-}\le t_{+}<(bt+1)/a
%\end{displaymath}
%implies 
%$t^{-1}b_2([t_{-}, t_{+}]\setminus\{t\})\le ct_{+}^{-2}b^*_2(t)$. 
%Therefore, (\ref{d-SSS}) follows from the bound $t^{-2}_{+}b^*_2(t)=o_P(1)$. In 
%order to prove this bound we show that for any $0<\varepsilon<1$ 
%\begin{equation}\label{S(k)}
%\PP(t^{-2}_{+}b^*_2(t)>\varepsilon)\le \varepsilon b_1+o(1).
%\end{equation}
%
%We write, see (\ref{d-bbb}), 
%$\PP(t^{-2}_{+}b^*_2(t)>\varepsilon)
%=
%\PP(\{t^{-2}_{+}b^*_2(t)>\varepsilon\}\cap {\cal B}_t(\varepsilon))+o(1)$ 
%and estimate
%\begin{displaymath}
%\PP(\{t^{-2}_{+}b^*_2(t)>\varepsilon\}\cap {\cal B}_t(\varepsilon))
%\le 
%\PP(t^{-1}_{+}b^*_1(t)>\varepsilon^{-1})
%\le 
%\varepsilon t^{-1}_{+}\E b^*_1(t) 
%\le
%\varepsilon b_1.
%\end{displaymath}
%Here the  second inequality  follows by  Markov's inequality. 

 Proof of (\ref{d-L22}). Let ${\overline{\cal A}}_t$ denote the complement event 
to ${\cal A}_t$.
 We have, by the union bound and Markov's inequality,
 \begin{displaymath}
 \PP_t({\overline {\cal A}}_t)
 \le
 \sum_{i\in T_t}\PP_t(\lambda_{it}\ge 1)
 \le 
 \sum_{i\in T_t}(it)^{-1}Y_t^2a_2
 \le c a_2t^{-1}Y_t^2.
\end{displaymath}
%In the last step we used the simple bound $\sum_{i\in T_t}i^{-1}\le c$. 
We obtain the bound
$\PP_t({\overline {\cal A}}_1)=o(1)$, which implies (\ref{d-L22}), see (\ref{fact}).

Proof of (\ref{d-L21}). In view of (\ref{fact}) it suffices to show that
$\PP_X(d(v_t)\not=L_t)=o_P(1)$.
We note that $d(v_t)\not= L_t$ if and only if $S\ge 1$, where 
$S=\sum'
{\mathbb I}_{i_1}
{\mathbb I}_{i_2}
{\mathbb I}_{i_1j}
{\mathbb I}_{i_2j}$.
Here we denote  
$\sum'=\sum_{\{i_1,i_2\}\subset T_t}\sum_{j\in T^*_{i_1}\cap T^*_{i_2}, \, j\not=t}$.
Observing that 
\begin{displaymath}
\E_X {\mathbb I}_{i_1}
{\mathbb I}_{i_2}
{\mathbb I}_{i_1j}
{\mathbb I}_{i_2j}
= 
\E_X p_{i_1t}p_{i_2t}p_{i_1j}p_{i_2j}
\le
a_2^2Y_t^2Y_j^2/(i_1i_2tj) 
\end{displaymath}
we obtain, by Markov's inequality,
\begin{equation}\label{SXS}
\PP_X(d(v_t)\not=L_t)
=
\PP_X(S\ge 1)
\le 
\E_X S
\le 
a_2^2Y_t^2t^{-1}\sum' Y_j^2(i_1i_2j)^{-1}.
%\le 
%c\, a_2^2Y_t^2t^{-1}b_2([t_{-},t_{+}]\setminus\{t\}).
\end{equation}
The simple bound $\sum_{\{i_1,i_2\}\subset T_t}\frac{1}{i_1i_2}\le c$ implies
$\sum' Y_j^2(i_1i_2j)^{-1}\le cb_2([t_{-},t_{+}]\setminus\{t\})$. 
Now, by (\ref{d-SSS}) the right-hand side of (\ref{SXS}) 
tends to zero in probability.

Proof of (\ref{d-L23}). Denote ${\hat X}_i=\max\{X_i, 1\}$, ${\hat Y}_j=\max\{Y_j, 1\}$,  and 
let ${\hat Q}_{XY}(t)$  denote the sum (\ref{QXY}), where $\lambda_{ij}$ is replaced by ${\hat\lambda}_{ij}={\hat X}_i{\hat Y}_j/\sqrt{ij}$.
We observe that  $\E X_1^2<\infty$ and $\E Y_1<\infty$ imply 
\begin{displaymath}
a_{\varphi}:=\E {\hat X}_1^2\varphi({\hat X}_1)<\infty,
\qquad
b_{\varphi}:=\E {\hat Y}_1\varphi({\hat Y}_1)<\infty,
\end{displaymath}
for some positive increasing function $\varphi:[1,+\infty)\to[0,+\infty)$ 
satisfying $\varphi(u)\to+\infty$ as $u\to+\infty$ (clearly, $\varphi(\cdot)$ 
depends on the distributions of $X_1$ and $Y_1$). In addition, we can choose 
$\varphi$ satisfying
 $\varphi(u)\le u$ and $\varphi(su)\le \varphi(s)\varphi(u)$, for $s,u\ge 1$. 
From these inequalities one  derives the inequality
 $\min\{1,{\hat \lambda}_{ij}\}
\le 
\varphi({\hat X}_i)\varphi({\hat Y}_j)/\varphi(\sqrt{ij})$. The latter inequality implies
 % see \cite{BloznelisD2012}
 \begin{displaymath}
 {\hat Q}_{XY}(t)
 \le 
 {\hat Y}_t
 Q^*_{XY}(t),
 \qquad
 Q^*_{XY}(t):=
 \sum_{i\in T_t}
 \frac{{\hat X}_i^2\varphi({\hat X}_i)}{\sqrt{ti}}
 \sum_{j\in T^*_i\setminus\{t\}}\frac{{\hat Y}_j\varphi({\hat Y}_j)} {\sqrt{ij}}
 \frac{1}{\varphi(\sqrt{ij})}
\end{displaymath}
Furthermore, for $i\in T_t$ and $j\in T_i^*$ we have  
$ij\ge \lfloor at\rfloor t_{-}=:t_*^2$, and $t_*\to+\infty$ as $t\to+\infty$. Hence
\begin{displaymath}
\E Q^*_{XY}(t)\le \frac{1}{\varphi(t_*)}
a_{\varphi}b_{\varphi}
\sum_{i\in T_t}\frac{1}{\sqrt{ti}}\sum_{j\in T^*_i}\frac{1}{\sqrt{ij}}
=
O\left(\frac{1}{\varphi(t_*)}\right)=o(1).
\end{displaymath}
This bound together with  the inequalities $Q_{XY}(t)\le {\hat Q}_{XY}(t)\le {\hat Y}_tQ^*_{XY}(t)$ shows (\ref{d-L23}).

Proof of (\ref{eb1}). These inequalities follow from (\ref{Sumos}).
% Let 
%$s_i=\min\{j: \lfloor bj\rfloor\ge i\}$ 
%and 
%$u_i=\max\{j: \lfloor aj\rfloor \le i\}$ denote  the endpoints of the interval $T_i^*$. 
%Clearly, $|s_i-ib^{-1}|\le c$ and $|u_i-ia^{-1}|\le c$.  
%It follows from the simple inequalities 
%\begin{displaymath}
%u^{1/2}-s^{1/2}\le 2^{-1}\sum_{s\le i\le u}i^{-1/2}\le u^{1/2}-s^{1/2}+s^{-1/2},
%\qquad
%s\le u,
%\qquad 
%s, u=2,3,\dots, 
%\end{displaymath} 
%that 
%\begin{equation}\label{|T|}
%2^{-1}\sum_{j\in T_i^*\setminus\{t\}}j^{-1/2}=i^{1/2}(a^{-1/2}-b^{-1/2})+r'_i,
%\qquad 
%|r'_i|\le ci^{-1/2}.
%\end{equation}
%The latter inequality implies (\ref{eb1}).

Proof of (\ref{eb2}). Denote  $\Delta=|{\tilde b}-b_1\gamma_2 i^{1/2}|$, where
${\tilde b}$ denotes the sum 
$b_1(T_i^*\setminus\{t\})$, but with  $Y_j$  replaced by 
${\tilde Y}_{j}=Y_j{\mathbb I}_{\{Y_j\le \varepsilon^2j\}}$, 
$j\in T_i^*\setminus\{t\}$.
We have  
\begin{equation}\label{eb2+}
|b_1(T_i^*\setminus\{t\})-b_1\gamma_2 i^{1/2}|{\mathbb I}_{{\cal B}_{t}(\varepsilon)}
=
\Delta{\mathbb I}_{{\cal B}_{t}(\varepsilon)}
\le
\Delta 
\le 
\Delta_1+\Delta_2+\Delta_3,
\end{equation}
where we denote
$\Delta_1=|{\tilde b}-\E{\tilde b}|$,
$\Delta_2=|\E{\tilde b}-\E b_1(T_i^*\setminus\{t\})|$,
$\Delta_3=|\E b_1(T_i^*\setminus\{t\})- b_1\gamma_2 i^{1/2}|$.
Next, we evaluate $\E \Delta_1$ and $\Delta_2$:
\begin{eqnarray}\label{Delta1}
&&
(\E\Delta_1)^2
\le
\E(\Delta_1^2)
\le 
\sum_{j\in T_i^*\setminus\{t\}}j^{-1}\E {\tilde Y}_j^2\le \varepsilon^2b_1|T_i^*|, 
\\
\label{Delta2}
&&
\Delta_2\le\sum_{j\in T_i^*\setminus\{t\}}j^{-1/2}\E Y_j{\mathbb I}_{\{Y_j>\varepsilon^2j\}}
\le 
\E Y_1{\mathbb I}_{\{Y_1>\varepsilon^2t_{-}\}}\sum_{j\in T_i^*\setminus\{t\}}j^{-1/2}.
\end{eqnarray}
In (\ref{Delta1}) we first apply Cauchy-Schwartz, then use the linearity of  variance of an iid sum, and finally apply the inequality
$\Var {\tilde Y}_j\le \E {\tilde Y}_j^2\le j^{-1}\varepsilon^2\E Y_j$.   
%In (\ref{Delta2}) we use the fact that $t_{-}\le j$, $j\in T_i^*$.
Invoking (\ref{eb1}), (\ref{Delta1}), (\ref{Delta2}) in (\ref{eb2+}) 
and using  (\ref{Sumos}) and $|T_i^*|\le ci$
we obtain (\ref{eb2}).

Proof of (\ref{eb2+A}).
We write $\E|a_1(T_t)-a_1\gamma_1t^{1/2}|\le \E {\tilde \Delta}_1+{\tilde\Delta}_2$, 
where 
${\tilde \Delta}_1:=|a_1(T_t)-\E a_1(T_t)|$ 
and ${\tilde \Delta}_2=|\E a_1(T_t)-a_1\gamma_1t^{1/2}|$, and invoke the inequalities
\begin{displaymath}
(\E{\tilde \Delta}_1)^2\le \E{\tilde \Delta}_1^2=\sum_{i\in T_t}j^{-1}(a_2-a_1^2)\le c a_2
\end{displaymath}
and ${\tilde \Delta}_2\le ca_1\le ca_2^{1/2}$, see (\ref{eb1}).

%Lemos įrodymo pabaiga
\end{proof}

Inequality (\ref{LeCam})  below
is  referred to as LeCam's inequality, see e.g., \cite{Steele}.
\begin{lem}\label{LeCamLemma} Let $S={\mathbb I}_1+ {\mathbb I}_2+\dots+ {\mathbb I}_n$ be the sum of independent random indicators
with probabilities $\PP({\mathbb I}_i=1)=p_i$. Let $\Lambda$ be Poisson random variable with mean $p_1+\dots+p_n$. The total variation
distance between the distributions $P_S$ of $P_{\Lambda}$ of $S$ and $\Lambda$
\begin{equation}\label{LeCam}
d_{TV}(S,\Lambda):=\sup_{A\subset \{0,1,2\dots \}}|\PP(S\in A)-\PP(\Lambda\in A)|
%= \frac{1}{2} \sum_{k\ge 0}|\PP(S=k)-\PP(\Lambda=k)|
\le
\sum_{i}p_i^2.
\end{equation}
\end{lem}

\begin{proof}[Proof of Theorem  \ref{T1}] 
%\end{proof}
Before the proof we introduce some notation.
Given $X,Y$, we generate independent Poisson random variables 
\begin{displaymath}
\eta_i, 
\quad
\xi_{1i},
\quad
\xi_{3i}, 
\quad
\xi_{4i},
\quad
\Delta_{ri},
\qquad
i\in T_t,
\quad
r=1,2,3,
\end{displaymath}
with conditional mean values 
\begin{eqnarray}\nonumber
&&
{\tilde \E}\eta_i=\lambda_{it},
\qquad
{\tilde \E}\xi_{1i}=\sum_{j\in T^*_i\setminus\{t\}}p_{ij},
\qquad
{\tilde \E}\xi_{3i}=X_ib_1\gamma_2,
\qquad
{\tilde \E}\xi_{4i}=X_i{\overline b}i^{-1/2},
\\
\nonumber
&&
{\tilde \E}\Delta_{1i}=\sum_{j\in T^*_i\setminus\{t\}}(\lambda_{ij}-p_{ij}),
\qquad
{\tilde \E}\Delta_{2i}=X_i\delta_{2i}i^{-1/2},
\qquad
{\tilde \E}\Delta_{3i}=X_i\delta_{3i}i^{-1/2}.
\end{eqnarray}
Here 
%$\gamma_2=2(a^{-1/2}-b^{-1/2})$ is  constant and the random variables
\begin{displaymath}
\delta_{2i}=b_1(T_i^*\setminus\{t\})-{\overline b},
\qquad
\delta_{3i}=b_1\gamma_2 i^{1/2}-{\overline b},
\qquad
{\overline b}=\min\{b_1(T_i^*\setminus\{t\}),\,  b_1\gamma_2 i^{1/2}\}.
\end{displaymath}
Finally, we define $\xi_{2i}=\xi_{1i}+\Delta_{1i}$, $i\in T_t$ and introduce random variables
\begin{equation}\label{LGL}
L_{0t}=\sum_{i\in T_t}\eta_iu_i,
\qquad
L_{rt}=\sum_{i\in T_t}\eta_i\xi_{ri},
\qquad
r=1,2,3.
\end{equation}
 We assume, in addition, that given $X,Y$ the families of random variables
$\{{\mathbb I}_i, i\in T_t\}$ and $\{\xi_{ri}, i\in T_t, r=1,2,3,4\}$ 
are conditionally independent, and that $\{\eta_i, i\in T_t\}$ is conditionally independent
of the set of edges of $H_{X,Y}$ that are not incident to $v_t$.

We are ready to start the proof. In view of (\ref{d-L21}) the random variables $d(v_t)$ 
and $L_t$ have the same asymptotic 
distribution (if any). We shall prove that $L_t$ converges in distribution to $d_*$. 
In the proof we  approximate $L_t$ by the random variable $L_{3t}$, see (\ref{DeltaT}) and
(\ref{L-L}) below.  Afterwards
 we show that $L_{3t}$ converges in 
distribution to $d_*$.

In order to show that$L_t$ and $L_{3t}$ have the same asymptotic distribution  
(if any) we prove the bounds
\begin{eqnarray}\label{DeltaT}
&&
d_{TV}(L_t, L_{0t})=o(1),
\qquad
d_{TV}(L_{0t},L_{1t})=o(1),
\\
\label{L-L}
&&
\E|L_{1t}-L_{2t}|=o(1),
\qquad
\qquad
{\tilde L}_{2t}-{\tilde L}_{3t}=o_P(1).
\end{eqnarray}
Here ${\tilde L}_{2t}$ and ${\tilde L}_{3t}$ are marginals of the random 
vector $({\tilde L}_{2t}, {\tilde L}_{3t})$ constructed in (\ref{marginal}) below which has the 
property that ${\tilde L}_{2t}$ has the same distribution as $L_{2t}$ and ${\tilde L}_{3t}$ has 
the same distribution as $L_{3t}$.

Let us prove the first bound of (\ref{DeltaT}). We shall show below that
\begin{equation}\label{TV1}
 {\tilde d}_{TV}(L_t, L_{0t}){\mathbb I}_{{\cal A}_t}\le t^{-1}Y_t^2a_2(T_t).
\end{equation} 
From the inequality $\E a_2(T_t)=\sum_{i\in T_t}i^{-1}a_2\le c\,a_2$ we conclude that
$a_2(T_t)$ is stochastically bounded.
Hence $t^{-1}Y_t^2 a_2(T_t)=o_P(1)$. This bound and (\ref{TV1}) combined with 
(\ref{d-L22}) imply 
\begin{displaymath}
{\tilde d}_{TV}(L_t, L_{0t})
\le 
{\tilde d}_{TV}(L_t, L_{0t}){\mathbb I}_{{\cal A}_t}+{\mathbb I}_{{\overline{\cal A}}_t}=o_P(1).
\end{displaymath}
Now the first bound of (\ref{DeltaT}) follows from (\ref{fact*}).
It remains to prove (\ref{TV1}).
We denote  
$L'_k=\sum_{i=\lfloor at\rfloor }^k{\mathbb I}_iu_i+\sum_{i=k+1}^{\lfloor bt\rfloor}\eta_iu_i$ and write, by the triangle inequality, 
\begin{displaymath}
 {\tilde d}_{TV}(L_t,L_{0t})\le \sum_{k\in T_t}{\tilde d}_{TV}(L'_{k-1}, L'_k).
\end{displaymath}
Then we estimate 
${\tilde d}_{TV}(L'_{k-1}, L'_k)
\le 
{\tilde d}_{TV}(\eta_k,{\mathbb I}_k)\le (kt)^{-1}Y_t^2X_k^2$.
Here the first inequality follows from the properties of the total variation distance. 
The second inequality follows from Lemma \ref{LeCamLemma} and the fact 
that on the event ${\cal A}_t$ we have
$p_{kt}=\lambda_{kt}$.

\medskip
%------------------------

Let us prove the second bound of (\ref{DeltaT}).
In view of (\ref{fact}) it suffices to  show that
  ${\tilde d}_{TV}(L_{0t},L_{1t})=o_P(1)$. For this purpose we write, 
by the triangle inequality, 
%of the total variation distance,
\begin{equation}\label{dTV4}
 {\tilde d}_{TV}(L_{0t},L_{1t})\le \sum_{k\in T_t}{\tilde d}_{TV}(L^*_{k-1}, L^*_k),
\end{equation}
where  
$L^*_k:=\sum_{i=\lfloor at\rfloor}^k\eta_iu_i+\sum_{i=k+1}^{\lfloor bt\rfloor}\eta_i\xi_{1i}$, and
estimate
\begin{equation}\label{dTV41}
{\tilde d}_{TV}(L^*_{k-1}, L^*_k)\le {\tilde d}_{TV}(\eta_ku_k,\eta_k\xi_{1k})
\le 
{\tilde \PP}(\eta_k\not=0){\tilde d}_{TV}(u_k,\xi_{1k}).
\end{equation}
Now, invoking the inequalities  
\begin{displaymath}
{\tilde \PP}(\eta_k\not=0)=1-e^{-\lambda_{kt}}\le \lambda_{kt}
\qquad
{\text{and}}
\qquad
{\tilde d}_{TV}(u_k,\xi_{1k})\le \sum_{j\in T_k^*\setminus\{t\}}p^2_{kj},
\end{displaymath}
 see (\ref{LeCam}), we obtain from 
(\ref{dTV4}), (\ref{dTV41})  and (\ref{d-L23}) that
\begin{displaymath}
 {\tilde d}_{TV}(L_0,L_1)\le Q_{XY}(t)=o_P(1).
\end{displaymath}

\medskip

Let us prove the first  bound of (\ref{L-L}).
We observe that 
\begin{displaymath}
 |L_{2t}-L_{1t}|=L_{2t}-L_{1t}
=
\sum_{i\in T_t}\eta_i\Delta_{1i}
\end{displaymath}
and 
\begin{displaymath}
{\tilde \E}\sum_{i\in T_t}\eta_i\Delta_{1i}
=
\sum_{i\in T_t}\lambda_{it}\sum_{j\in T^*_i\setminus\{t\}}
(\lambda_{ij}-1){\mathbb I}_{\{\lambda_{ij}>1\}}
\le Q_{XY}(t). 
\end{displaymath}
We obtain $\E|L_{2t}-L_{1t}|\le \E Q_{XY}(t)=o(1)$, see
 (\ref{d-L23}).

\medskip

Let us prove the second  bound of (\ref{L-L}).  We note  that the random vector
\begin{equation}\label{marginal}
({\tilde L}_{2t},{\tilde L}_{3t}),
\qquad
{\tilde L}_{2t}=\sum_{i\in T_t}\eta_i(\xi_{4i}+\Delta_{2i}), 
\qquad
{\tilde L}_{3t}=\sum_{i\in T_t}\eta_i(\xi_{4i}+\Delta_{3i})
\end{equation}
has the marginal distributions of $(L_{2t},L_{3t})$. 
In addition, since $\Delta_{2i}$, $\Delta_{3i}\ge 0$ and at most one of them
is non-zero, we have $|\Delta_{2i}-\Delta_{3i}|=\Delta_{2i}+\Delta_{3i}$.
Therefore, we can write
\begin{equation}\label{L2-L3}
{\tilde \Delta}:=|{\tilde L}_{2t}-{\tilde L}_{3t}|
\le 
\sum_{i\in T_t}|\eta_i||\Delta_{2i}-\Delta_{3i}|=\sum_{i\in T_t}\eta_i
(\Delta_{2i}+\Delta_{3i}).
\end{equation}
We remark that given $X,Y$ the random variable $\Delta_{2i}+\Delta_{3i}$ has 
Poisson distribution with (conditional) mean value 
\begin{displaymath}
 {\tilde \E} (\Delta_{2i}+\Delta_{3i})
=
X_i i^{-1/2} \delta_i,
\qquad
\delta_i:=|b_1(T_i^*\setminus\{t\})-b_1\gamma_2 i^{1/2}|. 
\end{displaymath}
Therefore, (\ref{L2-L3}) implies
${\tilde \E}{\tilde \Delta}
\le 
t^{-1/2} Y_t\sum_{i\in T_t}X_i^2i^{-1}\delta_i$. Next, for $0<\varepsilon<1$, we write
\begin{displaymath}
 \E{\mathbb I}_{{\cal B}_t(\varepsilon)}{\tilde \Delta}
\le 
\E{\mathbb I}_{{\cal B}_t(\varepsilon)}
t^{-1/2} Y_t\sum_{i\in T_t}X_i^2i^{-1}\delta_i
=
b_1a_2t^{-1/2}\sum_{i\in T_t}i^{-1}\E \delta_i{\mathbb I}_{{\cal B}_t(\varepsilon)}.
\end{displaymath}
Invoking  upper bound (\ref{eb2}) 
for $\E\delta_i{\mathbb I}_{{\cal B}_t(\varepsilon)}$ we obtain
$\E{\mathbb I}_{{\cal B}_t(\varepsilon)}{\tilde \Delta}\le cb_1^{3/2}a_2\varepsilon+
o(1)$. Finally, this bound combined with  Markov's inequality and (\ref{d-bbb}) yields
\begin{displaymath}
 \PP({\tilde \Delta}\ge 1)=
\PP(\{{\tilde \Delta}\ge 1\}\cap {\cal B}_t(\varepsilon))+o(1)
\le 
\E{\mathbb I}_{{\cal B}_t(\varepsilon)}{\tilde \Delta}+o(1)
\le 
cb_1^{3/2}a_2\varepsilon+
o(1).
\end{displaymath}
We conclude that $\PP({\tilde \Delta}\not=0)=\PP({\tilde \Delta}\ge 1)=o(1)$.

\medskip

%\end{document}

Next  we prove that $L_{3t}$ converges in distribution to $d_*$ defined by (\ref{d*1}). 
Let $Y_{\star}$ be a random copy of $Y_1$, which is independent of $X,Y$. 
Given $X,Y,Y_{\star}$, we generate independent Poisson random variables
$\eta^{\star}_{k}$, $k\in T_t$
with (conditional) mean values
$\E(\eta^{\star}_{k}|X,Y,Y_{\star})=\lambda_{k\star}$, 
where $\lambda_{k\star}=X_kY_{\star}(kt)^{-1/2}$. We assume that, given 
$X,Y,Y_{\star}$, 
the family of random variables $\{\eta^{\star}_k,\, k\in T_t\}$ is conditionally independent
of $\{\xi_{3k},\, k\in T_t\}$. 
Define $L^{\star}_t=\sum_{k\in T_t}\eta^{\star}_k\xi_{3k}$. We note that $L^{\star}_t$ is
defined
in the same way as $L_{3t}$ above, but with $Y_t$ replaced by $Y_{\star}$.
Let $d_{\star}$ be defined in the same way as $d_*$, but with $\lambda_1$ replaced by 
$\lambda_{\star}=Y_{\star}a_1\gamma_1$. 
Since $L_{3t}$  has the same distribution as $L^{\star}_t$, and $d_*$ has the
same distribution as $d_{\star}$, it suffices to show that 
$L^{\star}_t$ converges in distribution to $d_{\star}$.
For this purpose we 
 show the convergence of Fourier-Stieltjes transforms
  $\E e^{izL^{\star}_t}\to\E e^{izd_{\star}}$, 
for each $z\in (-\infty,+\infty)$. Denote 
$\Delta^{\star}(z)=e^{izL^{\star}_t}-e^{iz d_{\star}}$. We shall show below 
that, for any real $z$ and any realized value $Y_{\star}$ 
there exists a positive constant 
$c^{\star}=c^{\star}(z,Y_{\star})$ such that for every $0<\varepsilon<0.5$ 
we have
\begin{equation}\label{epsilon}
 \limsup_t|\E(\Delta^{\star}(z)|Y_{\star})|< c^{\star}\varepsilon.
\end{equation}
Clearly, (\ref{epsilon}) implies $\E(\Delta^{\star}(z)| Y_{\star})=o(1)$.
 This fact together with the simple inequality
 $|\Delta^{\star}(z)|\le 2$  yields $\E \Delta^{\star}(z)=o(1)$, 
by Lebesgue's dominated convergence theorem.
Finally, the identity  $\E \Delta^{\star}(z)=\E e^{izL^{\star}_t}-\E e^{izd_{\star}}$ 
implies
$\E e^{izL^{\star}_t}\to\E e^{izd_{\star}}$.

%---------------------------------------------------....................

We fix $0<\varepsilon<0.5$ and prove (\ref{epsilon}). 
Before the proof we introduce  some notation. Denote
\begin{eqnarray}\nonumber
 &&
f_\varkappa(z)=\E e^{iz\varkappa_1},
\qquad
{\bar f}_{\varkappa}(z)=\sum_{r\ge 0}e^{izr}{\bar p}_r, 
\qquad
{\bar p}_r
=
{\bar \lambda}^{-1}\sum_{k\in T_t}\lambda_{k\star}{\mathbb I}_{\{\xi_{3k}=r\}},
\qquad
{\bar \lambda}=\sum_{k\in T_t}\lambda_{k\star},
\\
\nonumber
&&
\delta=({\bar f}_\varkappa(z)-1){\bar \lambda}-(f_{\varkappa}(z)-1)\lambda_{\star},
\qquad
f(z)=\E_{\star}e^{izd_{\star}}, 
\qquad
{\bar f}(z)={\bar\E}e^{izL^{\star}_t}.
\end{eqnarray}
Here ${\bar \E}$ denotes the conditional expectation given $X,Y, Y_{\star}$ 
and $\{{\xi}_{3k},\, k\in T_t\}$. By $\E_{\star}$ we denote the conditional expectation given 
$Y_{\star}$.

%--------------------------------------------------------------------------------------------

Introduce the event
${\cal D}=\{|a_1(T_t)-\gamma_1a_1t^{1/2}|<\varepsilon t^{1/2}\min\{1, \gamma_1a_1\}\}$ 
and let ${\overline{\cal D}}$ 
denote the complement event.  
Furthermore, select the number 
$T>1/\varepsilon$  such that $\PP(\varkappa_1\ge T)< \varepsilon$. By 
$c^{\star}_1, c^{\star}_2, \dots$ 
we denote positive numbers which do not depend on $t$.

We observe that, given $Y_{\star}$, the conditional distribution of $d_{\star}$ is the 
compound Poisson distribution
with the characteristic function $f(z)=e^{\lambda_{\star}(f_{\varkappa}(z)-1)}$. Similarly,
given $X,Y, Y_{\star}$ 
and $\{\xi_{3k},\, k\in T_t\}$, 
the conditional distribution of $L^{\star}_t$ is the 
compound Poisson distribution
with the characteristic function 
${\bar f}(z)=e^{{\bar \lambda}({\bar f}_{\varkappa}(z)-1)}$.
In the proof of (\ref{epsilon}) we exploit the convergence
${\bar\lambda}\to\lambda_{\star}$ and ${\bar f}_{\varkappa}(z)\to f_{\varkappa}(z)$.

%-----------------------------------------------------------------------------

\medskip

Let us prove  (\ref{epsilon}). 
We write 
\begin{displaymath}
\E_{\star}\Delta^{\star}(z)=I_{1}+I_{2},
\qquad
I_{1}=\E_{\star}\Delta^{\star}(z){\mathbb I}_{\cal D},
\qquad
I_2=\E_{\star}\Delta^{\star}(z){\mathbb I}_{\overline{\cal D}}.         
\end{displaymath}
Here $|I_2|\le 2\PP_{\star}({\overline {\cal D}})=2\PP({\overline {\cal D}})=o(1)$. 
Indeed, the bound $\PP({\overline {\cal D}})=o(1)$ follows from (\ref{eb2+A}), 
by Markov's inequality.
Next we estimate $I_{1}$. Combining the identity $\E_{\star}\Delta^{\star}(z)=\E_{\star}f(z)(e^{\delta}-1)$
with the inequalities $|f(z)|\le 1$ and $|e^{s}-1|\le |s|e^{|s|}$, 
we obtain
\begin{equation}
 |I_1|\le \E_{\star}|\delta|e^{|\delta|}{\mathbb I}_{\cal D}
\le 
c^{\star}_1\E_{\star}|\delta|{\mathbb I}_{\cal D}.
\end{equation}
Here we estimated $e^{|\delta|}\le e^{6\lambda_{\star}}=:c^{\star}_1$ using the inequalities
 \begin{displaymath}
 |\delta|\le 2{\bar \lambda}+2\lambda_{\star},
\qquad
{\bar \lambda}=Y_{\star}t^{-1/2}a_1(T_t)\le 2\lambda_{\star}. 
 \end{displaymath}
We remark that the last  inequality holds  provided that  event ${\cal D}$ occurs.

%------------------------------------------------------------------------------------------------------

\medskip

Finally, we show that 
$\E_{\star}|\delta|{\mathbb I}_{\cal D}
\le 
(c^{\star}_2+c^{\star}_3\lambda_{\star}+c^{\star}_4\lambda_{\star})\varepsilon+o(1)$. 
To this aim we write
\begin{displaymath}
\delta
=
({\bar f}_{\varkappa}(z)-1)({\bar\lambda}-\lambda_{\star})+({\bar f}_{\varkappa}(z)-f_{\varkappa}(z))\lambda_{\star}
\end{displaymath}
and estimate $|\delta|\le 2|{\bar\lambda}-\lambda_{\star}|+\lambda_{\star}|{\bar f}_{\varkappa}(z)-f_{\varkappa}(z)|$.
The inequality, which holds on the event ${\cal D}$, $|{\bar\lambda}-\lambda_{\star}|\le Y_{\star}\varepsilon$ implies 
$\E_{\star}|{\bar\lambda}-\lambda_{\star}|{\mathbb I}_{\cal D}\le c_2^{\star}\varepsilon$ with $c_2^{\star}:=Y_{\star}$.
Next we show that
\begin{displaymath}
 \E_{\star}|{\bar f}_{\varkappa}(z)-f_{\varkappa}(z)|{\mathbb I}_{\cal D}\le (c^{\star}_3+c^{\star}_4)\varepsilon+o(1).
\end{displaymath}
We first  split 
\begin{displaymath}
 {\bar f}_{\varkappa}(z)-f_{\varkappa}(z)=\sum_{r\ge 0}e^{izr}({\bar p}_r-p_r)=R_1-R_2+R_3,
\end{displaymath}
and then estimate separately the terms
\begin{displaymath}
 R_1=\sum_{r\ge T}e^{izr}{\bar p}_r,
\qquad
R_2=\sum_{r\ge T}e^{izr}p_r,
\qquad 
R_3=\sum_{0\le r<T}e^{izr}({\bar p}_r-p_r).
\end{displaymath}
Here we denote
$p_r=\PP(\varkappa_1=r)$. The upper bound for $R_2$ follows by the choice of $T$ 
\begin{displaymath}
 |R_2|\le \sum_{r\ge T}p_r=\PP(\varkappa_1\ge T)< \varepsilon.
\end{displaymath}
Next,
combining the identities
${\bar p}_r=(a_1(T_t))^{-1}\sum_{k\in T_t}k^{-1/2}X_k{\mathbb I}_{\{\xi_{3k}=r\}}$
and
\begin{displaymath}
 \sum_{r\ge T}\sum_{k\in T_t}k^{-1/2}X_k{\mathbb I}_{\{\xi_{3k}=r\}}=\sum_{k\in T_t}k^{-1/2}X_k{\mathbb I}_{\{\xi_{3k}\ge T\}} 
 \end{displaymath}
with the inequality $a_1(T_t)\ge t^{1/2}a_1\gamma_1/2$, which holds on the event ${\cal D}$,
 we obtain
\begin{equation}\nonumber
 |R_1|{\mathbb I}_{\cal D}
 \le
\sum_{r\ge T}{\bar p}_r
\le
 \frac{2}{a_1\gamma_1 t^{1/2}}\sum_{k\in T_t}\frac{X_k}{k^{1/2}}{\mathbb I}_{\{\xi_{3k}\ge T\}}
\le 
\frac{2}{a_1\gamma_1 t^{1/2}}\sum_{k\in T_t}\frac{X_k \xi_{3k}}{Tk^{1/2}}.
\end{equation}
Now, the identity $\E_{\star} X_k \xi_{3k}=a_2b_1\gamma_2$ implies 
$\E_{\star}|R_1|{\mathbb I}_{\cal D}\le c^{\star}_4 T^{-1}\le c^{\star}_4\varepsilon$.

\medskip

Now we estimate $R_3$. 
We denote $p'_r=a_1(T_t)(a_1\gamma_1t^{1/2})^{-1}{\bar p}_r$
and observe that the inequality $|a_1(T_t)(a_1\gamma_1t^{1/2})^{-1}-1|\le \varepsilon$, which holds on the event ${\cal D}$, implies
\begin{displaymath}
 |\sum_{0\le r\le T}e^{itr}({\bar p}_r- p'_r)|{\mathbb I}_{\cal D}
 \le 
 \varepsilon\sum_{0\le r\le T}
{\bar p}_r\le \varepsilon.
\end{displaymath}
In the last inequality we use the fact that the probabilities  $\{{\bar p}_r\}_{r\ge 0}$ 
sum up to $1$. It follows now that
\begin{equation}\label{R3-M}
 |R_3|{\mathbb I}_{\cal D}\le \varepsilon+\sum_{0\le r\le T}|p'_r-p_r|.
\end{equation}
Next we estimate
\begin{equation}\label{R3-0}
\E_{\star}|p'_r-p_r|\le \E_{\star}|p'_r-\E_{\star}p'_r|+|\E_{\star}p'_r-p_r|
\end{equation}
where, by the Cauchy-Schwartz and the linearity of the variance of an iid sum, we have
\begin{eqnarray}\label{R3-1}
&&
(\E_{\star}|p'_r-\E_{\star}p'_r|)^2\le \E_{\star}|p'_r-\E_{\star}p'_r|^2\le (a_1\gamma_1t^{1/2})^{-2}a_2(T_t)\le ct^{-1} a_2a_1^{-2},
\\
\label{R3-2}
&&
|p_r-\E_{\star}p'_r|=p_r|1-(\gamma_1t^{1/2})^{-1}\sum_{k\in T_t}k^{-1/2}|\le ct^{-1}.
\end{eqnarray}
In (\ref{R3-1}) we first apply the Cauchy-Schwartz inequality, then use  the linearity of  variance and the simple inequality
$\Var X_k{\mathbb I}_{\{\xi_{3k}=r\}}\le a_2$. In (\ref{R3-2}) we use the identity
$\E_{\star}X_k{\mathbb I}_{\{\xi_{3k}=r\}}=a_1p_r$ and (\ref{Sumos}). 
From (\ref{R3-0}), (\ref{R3-1}), (\ref{R3-2}) we conclude that $\E_{\star}|p'_r-p_r|=O(t^{-1/2})$. Now (\ref{R3-M}) implies 
\begin{displaymath}
\E_{\star} |R_3|{\mathbb I}_{\cal D}
\le 
\varepsilon+O(|T|t^{-1/2})=\varepsilon+o(1).
\end{displaymath}
\end{proof}

\medskip

\subsection{Proof of Theorem \ref{T2}}

Here we assume that $\tau(t):=t^{\nu}$.  In the proof below we apply
the following simple 
approximations 
\begin{eqnarray}\label{Sumos2}
&&
\sum_{k\in T_t}k^{-(1-2\nu)/(2\nu)}=t^{1/2}\gamma'_1+rt^{(2\nu)^{-1}-1}, 
\qquad
\sum_{j\in T_k^*}j^{-1/2}=k^{(2\nu)^{-1}}\gamma'_2+r'k^{-(2\nu)^{-1}},
\\
\nonumber
&&
\gamma'_1:=2\nu(b^{(2\nu)^{-1}}-a^{(2\nu)^{-1}}),
\qquad
\qquad
\gamma'_2:=2(a^{-(2\nu)^{-1}}-b^{-(2\nu)^{-1}}),
\end{eqnarray}
where $|r|, |r'|\le c$. 
We also make use of relations (\ref{d-bbb}), (\ref{d-L21}), (\ref{d-L22}) and 
(\ref{d-L23}), which 
remain valid in the case where $\tau(t)=t^{\nu}$, and of the inequalities, for $k\in T_t$,
\begin{eqnarray}\label{nu1}
&& |\E b_1(T_k^*\setminus\{t\})-b_1\gamma'_2k^{1/(2\nu)}|\le cb_1k^{-1/(2\nu)},
\\
\label{nu2}
&&
\E|b_1(T_k^*\setminus\{t\})-b_1\gamma'_2 k^{1/(2\nu)}|{\mathbb I}_{{\cal B}_t(\varepsilon)}
\le 
ck^{1/(2\nu)}(\varepsilon b_1^{1/2} +\E Y_1{\mathbb I}_{\{Y_1>\varepsilon^2t_{-}\}})
+cb_1k^{-1/(2\nu)}.
\end{eqnarray}
We note that (\ref{nu1}) follows from the second identity of (\ref{Sumos2}), and (\ref{nu2})
is obtained  in the same way as (\ref{eb2}) above.

\begin{proof}[Proof of Theorem  \ref{T2}] 
Before the proof we introduce some notation.
Given $\varepsilon\in (0,1)$, 
denote
%introduce random variables
\begin{displaymath}
\zeta=\sum_{k\in T_t}\lambda_{kt}\zeta_{k},
\quad
\zeta_{k}=\beta_kb_1X_k{\mathbb I}'_{k},
\quad
{\mathbb I}'_{k}={\mathbb I}_{\{\beta_k b_1X_k<\varepsilon\}},
\quad
\beta_k=k^{(1-\nu)/(2\nu)}\gamma'_{2}.
\end{displaymath}
 Given $X,Y$, we generate independent Poisson random variables 
$\eta_k,{\hat \xi}_{3k}$, $k\in T_t$, 
with (conditional) mean values
 ${\tilde \E}\eta_k=\lambda_{kt}$, ${\tilde \E}{\hat \xi}_{3k}= \beta_kb_1X_k$ 
and independent 
 Bernoulli random variables
${\tilde {\mathbb I}}_{k}$, $k\in T_t$ with success probabilities 
\begin{displaymath}
{\tilde \PP}({\tilde {\mathbb I}}_{k}=1)=1-{\tilde \PP}({\tilde {\mathbb I}}_{k}=0)=\zeta_{k}.
\end{displaymath}
We assume that, given $X,Y$, 
the sequences 
$\{{\mathbb I}_k, k\in T_t\}$, 
$\{{\tilde {\mathbb I}}_{k}, k\in T_t\}$, $\{\eta_{k}, k\in T_t\}$,
$\{{\hat \xi}_{3 k}, k\in T_t\}$
are conditionally independent. 
Next, we introduce random variables
\begin{displaymath}
{\hat L}_{3t}=\sum_{k\in T_t}\eta_k{\hat \xi}_{3 k},
\quad
L_{4t}=\sum_{k\in T_t}{\mathbb I}_k{\hat \xi}_{3 k},
\quad
L_{5t}=\sum_{k\in T_t}{\mathbb I}_k{\mathbb I}'_{k}{\hat \xi}_{3 k},
\quad
L_{6t}=\sum_{k\in T_t}{\mathbb I}_k{\tilde {\mathbb I}}_{k}.
\end{displaymath}
Furthermore, we define the random variable $L_{7t}$ as follows. We first generate $X,Y$. 
Then, given $X,Y$, we
generate a Poisson random variable with the conditional mean value 
$\zeta$. The realized value of the Poisson random variable is denoted $L_{7t}$. Thus, we have
$\PP(L_{7t}=r)=\E e^{-\zeta}\zeta^r/r!$, for $r=0,1,\dots$.

Now we are ready to prove Theorem \ref{T2}. In the first step of the proof
we show that random variables $d(v_t)$ and ${\hat L}_{3t}$ have the same asymptotic 
distribution (if any). Here we proceed as in the proof of (\ref{DeltaT}), (\ref{L-L}) above 
and make use of 
(\ref{d-bbb}), (\ref{d-L21}), 
(\ref{d-L22}), (\ref{d-L23}), 
(\ref{nu1}), (\ref{nu2}).
In the second step  we show
that ${\hat L}_{3t}$ converges in distribution to $\Lambda_3$. For this purpose we prove that
\begin{eqnarray}\label{TL1}
&&
d_{TV}({\hat L}_{3t}, L_{4t})=o(1),
\qquad
\E(L_{4t}-L_{5t})=o(1),
\\
\label{TL2}
&&
d_{TV}(L_{6t},L_{7t})=o(1),
\qquad
\E e^{izL_{7t}}-\E e^{iz\Lambda_3}=o(1),
\end{eqnarray}
for every $-\infty<z<+\infty$, 
and that there exists $c>0$, depending only on $a,b,\nu$, such that 
for any $\varepsilon\in (0,1)$ we have
\begin{equation}\label{TL3}
d_{TV}(L_{5t}, L_{6t})\le c a_2b_1^2 \varepsilon.
\end{equation}

Let us prove (\ref{TL1}), (\ref{TL2}), (\ref{TL3}). The first bound of (\ref{TL1}) 
is obtained in the same way as the first bound of (\ref{DeltaT}).
 To show the second 
bound  of (\ref{TL1}) we write
\begin{displaymath}
{\tilde \E}(L_{t5}-L_{t6})
=
\sum_{k\in T_t}(1-{\mathbb I}_k'){\tilde \E}{\mathbb I}_{kt}{\tilde \E}{\hat \xi}_{3k}
=
Y_tb_1t^{-1/2}\sum_{k\in T_t}(1-{\mathbb I}_k')X_k^2\beta_kk^{-1/2}
\end{displaymath}
and apply the simple inequality 
\begin{equation}\label{truncated}
 \E X_k^2(1-{\mathbb I}_k')\le \E  X_{\underline t}^2(1-{\mathbb I}_{\underline t}'),
 \qquad
k\in T_t. 
\end{equation}
Here we denote ${\underline t}=\min\{k:\, k\in T_t\}$. 
We obtain
\begin{displaymath}
\E(L_{4t}-L_{5t})
=
\E {\tilde \E}(L_{4t}-L_{5t})
\le 
S_t b_1^2 \E  X_{\underline t}^2(1-{\mathbb I}_{\underline t}')=o(1).
\end{displaymath}
Here we denote $S_t=t^{-1/2}\sum_{t\in T_t}\beta_kk^{-1/2}$
and use the simple inequality $S_t\le c$. Furthermore, we invoke the bound
$\E  X_{\underline t}^2(1-{\mathbb I}_{\underline t}')=o(1)$, which holds since
 ${\underline t}\to+\infty$ as $t\to+\infty$ 

Let us prove (\ref{TL3}).
Proceeding as in (\ref{dTV4}), (\ref{dTV41}) and using the 
identity ${\tilde {\mathbb I}}_k={\tilde {\mathbb I}}_k{\mathbb I}_k'$  we write
\begin{displaymath}
 {\tilde d}_{TV}(L_{5t}, L_{6t})
\le 
\sum_{k\in T_t}
{\mathbb I}_k'
{\tilde \PP}({\mathbb I}_k\not=0)
{\tilde d}_{TV}({\hat \xi}_{3k}, {\tilde {\mathbb I}}_k).
\end{displaymath}
Next, we estimate 
${\mathbb I}_k'{\tilde d}_{TV}({\hat \xi}_{3k}, {\tilde {\mathbb I}}_k)\le \zeta_k^2$, by
LeCam's inequality (\ref{LeCam}), and invoke the inequality
${\tilde \PP}({\mathbb I}_k\not=0)\le \lambda_{kt}$. 
We obtain
\begin{displaymath}
 {\tilde d}_{TV}(L_{5t}, L_{6t})
\le 
\sum_{k\in T_t}
{\mathbb I}_k'\lambda_{kt}\zeta_k^2
\le 
\varepsilon\sum_{k\in T_t}
\lambda_{kt}\zeta_k.
\end{displaymath}
Here we estimated $\zeta_k^2\le \varepsilon\zeta_k$. Now the inequalities
\begin{displaymath}
d_{TV}(L_{5t},L_{6t})\le \E{\tilde d}_{TV}(L_{5t},L_{6t})
\le 
\varepsilon \sum_{k\in T_t}
\E\lambda_{kt}\zeta_k
\le a_2b_1^2S_t\varepsilon
\end{displaymath}
and $S_t\le c$ imply  (\ref{TL3}).

Let us prove the first relation of (\ref{TL2}).
In view of  
(\ref{fact}) it suffices to show that ${\tilde d}_{TV}(L_{6t},L_{7t})=o_P(1)$.
For this purpose we write 
\begin{displaymath}
{\tilde d}_{TV}(L_{6t},L_{7t})\le 
{\mathbb I}_{{\cal A}_1}{\tilde d}_{TV}(L_{6t},L_{7t})+{\mathbb I}_{{\overline {\cal A}}_1},
\end{displaymath}
 where
${\mathbb I}_{{\overline {\cal A}}_1}=o_P(1)$, see  (\ref{d-L22}), and estimate 
using  LeCam's inequality (\ref{LeCam})
\begin{displaymath}
{\mathbb I}_{{\cal A}_1}{\tilde d}_{TV}(L_{6t},L_{7t})
\le 
{\mathbb I}_{{\cal A}_1}
\sum_{k\in T_t}{\tilde \PP}^2({\mathbb I}_k{\tilde {\mathbb I}}_k=1){\mathbb I}_k'
\le 
Y_t^2b_1^2t^{-1}\sum_{k\in T_t}k^{-1}\beta_k^2X_k^4=o_P(1)
\end{displaymath}
Here we used 
the simple inequality
$t^{-1}\sum_{k\in T_t}k^{-1}\beta_k^2X_k^4\le c t^{-2\nu}\sum_{k\le bt^{\nu}}X_k^4$
and 
the fact that 
$\E X_1^2<\infty$ implies the bound $n^{-2}\sum_{k\le n}X_k^4=o_P(1)$, 
as $n\to+\infty$.

\medskip

Finally, we show the second relation of (\ref{TL2}). We write
${\tilde \E} e^{izL_{7t}}=e^{\zeta(e^{iz}-1)}$ and use the bound 
\begin{equation}\label{Y1gamma}
 Y_tb_1a_2\gamma-\zeta=o_P(1).
\end{equation}
We note that,  for any real $z$, the function 
$u\to e^{u(e^{iz}-1)}$ is bounded and uniformly continuous for $u\ge 0$. 
Therefore, (\ref{Y1gamma}) implies the convergence 
\begin{displaymath}
 \E e^{izL_{7t}}
=
\E e^{\zeta(e^{iz}-1)}
\to
\E e^{Y_tb_1a_2\gamma(e^{iz}-1)}
=
\E e^{iz\Lambda_3}.
\end{displaymath}
It remains to prove (\ref{Y1gamma}). We note that (\ref{truncated}) implies
\begin{equation}\label{vakaras1}
 \zeta=Y_tb_1\gamma_2''t^{-1/2}\sum_{k\in T_t}X_k^{2}k^{(2\nu)^{-1}-1}.
\end{equation}
Next, we split $\gamma=\gamma_1'\gamma_2''$ and invoke the expression for  
$\gamma_1'$ 
obtained from (\ref{Sumos2}). We obtain
\begin{equation}\label{vakaras2}
 Y_tb_1a_2\gamma=Y_tb_1\gamma_2''
t^{-1/2}\sum_{k\in T_t}k^{-(1-2\nu)/(2\nu)}a_2+o_P(1).
\end{equation}
We observe that (\ref{Y1gamma}) follows from (\ref{vakaras1}), 
(\ref{vakaras2}) and the bound
\begin{equation}\label{vakaras3}
 R:=t^{-1/2}\sum_{k\in T_t}(a_2-X_k^{2})k^{(2\nu)^{-1}-1}=o_P(1).
\end{equation}
In the proof of (\ref{vakaras3}) we use the standard truncation argument.
Let $\varepsilon>0$ and let ${\hat R}$ be defined as $R$ above, but with $X_k^2$ replaced 
by 
${\hat X}_k^2=X_k^2{\mathbb I}_{\{X_k^2<\varepsilon^2 k\}}$ and $a_2$ replaced by 
$\E {\hat X}_k^2$. We have $R={\hat R}+o_P(1)$ and 
$\PP({\hat R}>\varepsilon^{1/2})\le \varepsilon^{-1}\E{\hat R}^2\le c\varepsilon$.
Letting $\varepsilon\to 0$ we obtain $R=o_P(1)$.

\end{proof}

%++++++++++++++++++++++++++++++++++++++++++++++++++++++++++++++++++++++++++++++++++++++++++++++++++++++++++++++++++++++++++++=

\subsection{Proof of Theorem 3}

Before the proof we state an auxiliary lemma.

\begin{lem}\label{momentai}  Denote ${\bf I}_i^x={\mathbb I}_{\{X_i>i^{1/2}\}}$ 
and 
${\bf I}_j^y={\mathbb I}_{\{Y_j>j^{1/2}\}}$. 
We have
\begin{equation}\label{LAssort-1}
\lambda_{ij}(1-{\bf I}_i^x-{\bf I}_j^y)
\le 
\min\{1,\lambda_{ij}\}
\le 
\lambda_{ij}
\end{equation}
\end{lem}

\begin{proof}[Proof of Lemma \ref{momentai}]
The inequality ${\mathbb I}_{\{\lambda_{ij}> 1\}}\le {\bf I}_i^x+{\bf I}_j^y$
implies
\begin{displaymath}
\lambda_{ij}(1-{\bf I}_i^x-{\bf I}_j^y)
\le
 \lambda_{ij}-(\lambda_{ij}-1){\mathbb I}_{\{\lambda_{ij}> 1\}}
=
\min\{1,\lambda_{ij}\}.
\end{displaymath}
\end{proof}

\begin{proof}[Proof of Theorem \ref{T3}]
The proof of (\ref{alpha-s}), (\ref{alpha-t}), (\ref{alpha-u}) is very much the same.
Therefore, we only prove (\ref{p(Delta)}) and (\ref{alpha-t}). 

Before the proof we introduce some notation. Denote 
\begin{displaymath}
T_{st}=T_s\cap T_t,
\qquad
T_{tu}=T_t\cap T_u,
\qquad
T_{stu}=T_s\cap T_t\cap T_u,
\qquad
T=T_s\cup T_t\cup T_u.
\end{displaymath}
An attribute $w_i$ is called  witness of the edge $v_j\sim v_k$ 
whenever ${\mathbb I}_{ij}{\mathbb I}_{ik}=1$. In this case we say that witness
$w_i$ realizes the edge $v_j\sim v_k$.
Let $\Delta_1
=
\{\exists i:\,  {\mathbb I}_{is}{\mathbb I}_{it}{\mathbb I}_{iu}=1\}$ denote 
the event that all three edges 
of the triangle $v_s,v_t,v_u$ are realized by a common witness.
Let $\Delta_2$ denote the event that all three edges are realized by different witnesses,
\begin{displaymath}
\Delta_2
=
\{\exists \ \ {\text{distinct}} \ \  i,j,k \ \ {\text{such that }}
\ \ 
{\mathbb I}_{is}{\mathbb I}_{it}=1,
\ \  
{\mathbb I}_{js}{\mathbb I}_{ju}=1,
\ \ 
{\mathbb I}_{kt}{\mathbb I}_{ku}=1
\}.
\end{displaymath}
Let $\Delta=\{v_s\sim v_t$, $v_s\sim v_u$, $v_t\sim v_u\}$ 
denote the event  that vertices $v_s,v_t,v_u$ make up a triangle.
Introduce events ${\cal H}_t=\{v_s\sim v_t, v_t\sim v_u\}$ and
${\cal K}_t
=
\{
\exists i\not=j:\, {\mathbb I}_{it}{\mathbb I}_{is}{\mathbb I}_{jt}{\mathbb I}_{ju}=1\}$,
and  random variables
\begin{eqnarray}\nonumber
&& 
S=\sum_{au\le k\le bs}{\mathbb I}_{ks}{\mathbb I}_{kt}{\mathbb I}_{ku},
\qquad
Q=\sum_{au\le i<j\le bs}{\mathbb I}_{is}{\mathbb I}_{it}{\mathbb I}_{iu}
{\mathbb I}_{js}{\mathbb I}_{jt}{\mathbb I}_{ju},
\\
\nonumber
&&
S_t
=
\sum_{(i,j)\in I}
{\mathbb I}_{it}{\mathbb I}_{is}{\mathbb I}_{jt}{\mathbb I}_{ju},
\qquad
Q_t
=
\sum_{(i,j)\in I}
\
\
\sum_{(k,r)\in I, (k,r)\not=(i,j)}
{\mathbb I}_{is}{\mathbb I}_{it}{\mathbb I}_{jt}{\mathbb I}_{ju}
{\mathbb I}_{kt}{\mathbb I}_{ks}{\mathbb I}_{rt}{\mathbb I}_{ru}.
\end{eqnarray}
Here $I$ denote the set of all ordered pairs $(i,j)\in T\times T$ such that $i\not=j$. We remark that every $(i,j)$  indicates
a pair $(w_i,w_j)$ of possible witnesses
 of edges $v_s\sim v_t$ and $v_t\sim v_u$ respectively.

We note that for  $0<s<t<u$
satisfying
$\lceil au\rceil\le \lfloor bs\rfloor$ the ratios $t/s, u/t, u/s\in [1, b/a]$.
Hence the variables  $s,t,u\to+\infty$ are of the same order of magnitude.
\bigskip

Let us prove (\ref{p(Delta)}). 
We observe that $\Delta_1\subset \Delta\subset \Delta_1\cup \Delta_2$. Hence
\begin{equation}\label{Delta}
 \PP(\Delta_1)\le \PP(\Delta)\le \PP(\Delta_1)+\PP(\Delta_2). 
\end{equation}
Next, by inclusion exclusion, we write $S-Q\le {\mathbb I}_{\Delta_1}\le S$ and estimate
\begin{equation}\label{Delta11}
 \E S-\E Q\le \PP(\Delta_1)\le \E S.
\end{equation}
Finally, combining (\ref{Delta}) and (\ref{Delta11}) with the relations
\begin{eqnarray}\label{clust-S}
 \E S
&
=
&
\sum_{au\le k\le bs}\frac{\E X_k^3Y_sY_tY_u}{k^{3/2}\sqrt{stu}}+o(t^{-2})
=
\frac{a_3b_1^3}{\sqrt{stu}}
\left(\frac{2}{\sqrt{au}}-\frac{2}{\sqrt{bs}}\right)+o(t^{-2}),
\\
\nonumber
 \E Q
&
\le
&
\sum_{au\le i<j\le bs}\E \lambda_{is}\lambda_{it}\lambda_{iu}
\lambda_{js}\lambda_{jt}\lambda_{ju}
\le 
\frac{a_3^2b_2^3}{stu}\sum_{au\le i<j\le bs}\frac{1}{i^{3/2}j^{3/2}}
=O(t^{-4}),
\\
\label{Delta222}
\PP(\Delta_2)
&
\le
& 
\E \sum_{i,j,k\in T,\, i\not=j\not=k}
{\mathbb I}_{is}{\mathbb I}_{it}
{\mathbb I}_{js}{\mathbb I}_{ju}
{\mathbb I}_{kt}{\mathbb I}_{ku}
\le 
\frac{a_2^3b_2^3}{stu}
 \left(\sum_{i\in T}i^{-1}\right)^3=O(t^{-3}),
\end{eqnarray}
we obtain asymptotic expression (\ref{p(Delta)}) for $p_{\Delta}=\PP(\Delta)$.
We note that in the first step of (\ref{clust-S}) we apply Lemma \ref{momentai}, and in the last step of (\ref{Delta222}) we use the inequality
$\sum_{i\in T}i^{-1}\le c$.

\bigskip

Let us prove (\ref{alpha-t}). We note that (\ref{alpha-t}) follows from (\ref{p(Delta)}) and the relation
\begin{equation}
\label{path}
\PP({\cal H}_t)
=
\PP(\Delta)+a_2^2b_1^2b_2\frac{1} {t\sqrt{su}}\delta_{t|su}+o(t^{-2}).
\end{equation}
It remains to show (\ref{path}).
From the identity ${\cal H}_t=\Delta_1\cup{\cal K}_t$ we obtain
\begin{equation}\label{HHt1}
 \PP({\cal H}_t)
 =
 \PP(\Delta_1)+\PP({\cal K}_t)-\PP(\Delta_1\cap{\cal K}_t).
\end{equation}
Next, by inclusion exclusion, we write 
$S_t-Q_t\le {\mathbb I}_{{\cal K}_t}\le S_t$. These inequalities imply 
\begin{equation}\label{KTTT}
\E S_t-\E S_t(1-{\mathbb I}_{{\cal D}_{\varepsilon}})-\E Q_t{\mathbb I}_{{\cal D}_{\varepsilon}}
\le
\E{\mathbb I}_{{\cal K}_t}{\mathbb I}_{{\cal D}_{\varepsilon}}
\le 
\PP({\cal K}_t)
\le 
\E S_t.
\end{equation}
Here the event ${\cal D}_{\varepsilon}=\{Y_t\le \varepsilon t\}$ and  $\varepsilon\in(0,1)$ is non-random. In the remaining part 
of the proof we show that
\begin{eqnarray}\label{KSQ1}
&& 
\E S_t 
=
a_2^2b_1^2b_2\frac{1} {t\sqrt{su}}\delta_{t|s,u}+o(t^{-2}),
\\
\label{KSQ3}
&&
\PP(\Delta_1\cap{\cal K}_t)
=
O(t^{-3}),
\end{eqnarray}
and that there exists
 $c^*>0$ which does not depend on $s,t,u$ and $\varepsilon$ such that, for any $\varepsilon\in (0,1)$, 
\begin{equation}
\label{KSQ2***}
\E Q_t{\mathbb I}_{{\cal D}_{\varepsilon}}
\le 
c^*\varepsilon t^{-2}+O(t^{-3}),
\qquad
\E S_t(1-{\mathbb I}_{{\cal D}_{\varepsilon}})
=
o(t^{-2}).
\end{equation}
We observe that  (\ref{path}) follows from 
(\ref{HHt1}), (\ref{KTTT}),  (\ref{KSQ1}) and the bounds (\ref{KSQ3}), (\ref{KSQ2***}).
  
Let us prove (\ref{KSQ1}). Since the product
${\bar p}_{ij}:=p_{is}p_{it}p_{jt}p_{ju}$ is non zero whenever
$i\in T_{st}$ and $j\in T_{tu}$, we have
\begin{equation}\label{Stt}
 \E S_t
=
\E\sum_{(i,j)\in I}{\bar p}_{ij}
=
\E\sum_{(i,j): i\in T_{st}, j\in T_{tu}, \, i\not=j}{\bar p}_{ij}.
\end{equation}
It is convenient to split the set 
$\{(i,j): i\in T_{st}, j\in T_{tu}, \, i\not=j\}= {\mathbb T}_1\cup\cdots\cup  {\mathbb T}_4$
where 
\begin{eqnarray}\nonumber
&&
 {\mathbb T}_1=(T_{st}\setminus T_u)\times T_{tu},
\qquad
{\mathbb T}_2=T_{stu}\times (T_{tu}\setminus T_s),
\\
\nonumber
&&
{\mathbb T}_3=\{(i,j): \, i,j\in T_{stu}, i<j\},
\quad
{\mathbb T}_4=\{(i,j): \, i,j\in T_{stu}, j<i\}.
\end{eqnarray}
and write sum (\ref{Stt}) in the form 
\begin{equation}\label{Sttt}
{\tilde \E}S_t=S_{t1}+\dots+ S_{t4},
\qquad
S_{tk}
:=
\sum_{(i,j)\in {\mathbb T}_k}{\bar p}_{ij}.
\end{equation}
Now (\ref{KSQ1}) follows from (\ref{Sttt}) and the relations, for $1\le k\le 4$,
\begin{eqnarray}\label{KSQ1+}
&& 
\E S_{tk}
=
\E\sum_{(i,j)\in {\mathbb T}_k} \lambda_{is}\lambda_{it}\lambda_{jt}\lambda_{ju}+o(t^{-2})
=
a_2^2b_1^2b_2\frac{1} {t\sqrt{su}}\sum_{(i,j)\in {\mathbb T}_k}\frac{1}{ij}+o(t^{-2}),
\\
\nonumber
&&
\sum_{1\le k\le 4}\sum_{(i,j)\in {\mathbb T}_k}\frac{1}{ij}
=
\delta_{t|su}+O(t^{-1}).
\end{eqnarray}
In the first step of (\ref{KSQ1+}) we used Lemma \ref{momentai}.

Let us prove the first bound of  (\ref{KSQ2***}). 
We split the collection  of vectors$(i,j,k,r)$
\begin{displaymath}
{\mathbb Q}=\bigl\{(i,j,k,r)\in T^4
\
\
{\text{ such that}}
\
\ 
i\not=j, k\not=r 
\ \
{\text{and}}
\
\
(i,j)\not=(k,r)\bigr\}
\end{displaymath}
 into five non intersecting  pieces
${\mathbb Q}={\mathbb Q}_1\cup\cdots\cup{\mathbb Q}_5$, where
\begin{eqnarray}\nonumber
&&
{\mathbb Q}_1=\bigl\{(i,j,k,r): i=k\bigr\}\cap {\mathbb Q},
\qquad
{\mathbb Q}_2=\bigl\{(i,j,k,r): i=r\bigr\}\cap {\mathbb Q},
\\
\nonumber
&&
{\mathbb Q}_3=\bigl\{(i,j,k,r): j=k\bigr\}\cap {\mathbb Q},
\qquad
{\mathbb Q}_4=\bigl\{(i,j,k,r): j=r\bigr\}\cap {\mathbb Q},
\end{eqnarray}
and ${\mathbb Q}_5=\bigl\{(i,j, k,r):$ all $i,j,k,r$ are distinct $\bigr\}\cap {\mathbb Q}$, and write
\begin{displaymath}
Q_t
=\sum_{1\le z\le 5}Q_{tz},
\qquad
Q_{tz}=
\sum_{(i,j, k,r)\in {\mathbb Q}_z}{\mathbb I}_{is}{\mathbb I}_{it}{\mathbb I}_{jt}{\mathbb I}_{ju}
{\mathbb I}_{ks}{\mathbb I}_{kt}{\mathbb I}_{rt}{\mathbb I}_{ru}.
\end{displaymath}
Denote  ${\tilde {\mathbb Q}}=\{(i,j,r)\in T^3:$ all $i,j,r$ are distinct$\}$. 
Observing that the typical summand of the sum $Q_{t1}$ is  
${\mathbb I}_{is}{\mathbb I}_{it}{\mathbb I}_{jt}{\mathbb I}_{ju}
{\mathbb I}_{rt}{\mathbb I}_{ru}$ (since $i=k$), we write
\begin{eqnarray}
\nonumber
\E Q_{t1}{\mathbb I}_{{\cal D}_{\varepsilon}}
&
\le 
&
\E \sum_{(i,j,r)\in {\tilde {\mathbb Q}}}\lambda_{is}\lambda_{it}\lambda_{jt}\lambda_{ju}\lambda_{rt}\lambda_{ru}
{\mathbb I}_{{\cal D}_{\varepsilon}}
\\
\nonumber
&
\le 
&
 \frac{a_2^3}{s^{1/2}t^{3/2}u}\E Y_sY_t^3Y_u^2{\mathbb I}_{{\cal D}_{\varepsilon}}\left(\sum_{i\in T}\frac{1}{i}\right)^3
\\
\nonumber
&
\le
&
c^3\varepsilon \frac{a_2^3}{s^{1/2}t^{1/2}u}\E Y_sY_t^2Y_u^2
\\
\nonumber
&
\le
&
c'\varepsilon t^{-2}.
\end{eqnarray}
Here used  inequalities $Y_tt^{-1}{\mathbb I}_{{\cal D}_{\varepsilon}}\le \varepsilon$ and  $\sum_{i\in T}\frac{1}{i}\le c$.
 Similarly, we prove the inequality $\E Q_{t4}{\mathbb I}_{{\cal D}_{\varepsilon}}\le c'\varepsilon t^{-2}$.
Furthermore, observing that the typical summand of the sum $Q_{t2}$ is  
\linebreak
${\mathbb I}_{is}{\mathbb I}_{it}{\mathbb I}_{iu}{\mathbb I}_{jt}{\mathbb I}_{ju}
{\mathbb I}_{ks}{\mathbb I}_{kt}$ (since $i=r$), we write
\begin{eqnarray}
\nonumber
\E Q_{t2}{\mathbb I}_{{\cal D}_{\varepsilon}}
&
\le 
&
\E \sum_{(i,j,k)\in {\tilde {\mathbb Q}}}\lambda_{is}\lambda_{it}\lambda_{iu}\lambda_{jt}\lambda_{ju}\lambda_{ks}\lambda_{kt}
{\mathbb I}_{{\cal D}_{\varepsilon}}
\\
\nonumber
&
\le 
&
 \frac{a_3a_2^2}{st^{3/2}u}\E Y^2_sY_t^3Y_u^2{\mathbb I}_{{\cal D}_{\varepsilon}}
 \left(\sum_{i\in T}\frac{1}{i}\right)^2
\left( \sum_{i\in T}\frac{1}{i^{3/2}}\right)
\\
\nonumber
&
\le 
&
c^3 \frac{a_3a_2^2}{stu}\E Y^2_sY_t^2Y_u^2.
\end{eqnarray}
In the last step we used inequalities $Y_tt^{-1}{\mathbb I}_{{\cal D}_{\varepsilon}}\le 1$ and  $\sum_{i\in T}\frac{1}{i^{3/2}}\le ct^{-1/2}$.
Hence, $\E Q_{t2}{\mathbb I}_{{\cal D}_{\varepsilon}}=O(t^{-3})$. Similarly, we prove the bound $\E Q_{t3}{\mathbb I}_{{\cal D}_{\varepsilon}}=O(t^{-3})$. Finally, we estimate
 \begin{eqnarray}
 \nonumber
 \E Q_{t5}{\mathbb I}_{{\cal D}_{\varepsilon}}
& 
\le
& 
 \E\sum_{(i,j,k,r)\in {\mathbb Q}_5}\lambda_{is}\lambda_{is}\lambda_{jt}\lambda_{ju}\lambda_{ks}\lambda_{kt}\lambda_{rt}\lambda_{ru}
 {\mathbb I}_{{\cal D}_{\varepsilon}}
\\
\nonumber
&
\le 
&
\frac{a_2^4}{st^2u}\E Y_s^2Y_t^4Y_u^2{\mathbb I}_{{\cal D}_{\varepsilon}}\left(\sum_{i\in T}\frac{1}{i}\right)^{4}
\\
\nonumber
&
\le 
&
c'\varepsilon^2 t^{-2}.
\end{eqnarray}
In the last step we used the inequality $Y_t^2t^{-2}{\mathbb I}_{{\cal D}_{\varepsilon}}\le \varepsilon^2$.

Let us prove the second bound of  (\ref{KSQ2***}).  We have
\begin{displaymath}
\E S_t(1-{\mathbb I}_{{\cal D}_{\varepsilon}})
\le
\E\sum_{i,j\in T,\, i\not=j}\lambda_{is}\lambda_{it}\lambda_{jt}\lambda_{ju}(1-{\mathbb I}_{{\cal D}_{\varepsilon}})
\le \frac{a_2^2b_1^2}{st}\E Y_t^2{\mathbb I}_{\{Y_t\ge \varepsilon t\}}\left(\sum_{i\in T}i^{-1}\right)^2
=o(t^{-2}).
\end{displaymath}

Let us prove (\ref{KSQ3}). The inequalities ${\mathbb I}_{{\cal K}_t}\le S_t$, 
${\mathbb I}_{\Delta_1}\le S$ and $S\le {\tilde S}$, where
\begin{displaymath}
{\tilde S}=\sum_{k\in T}{\mathbb I}^*_k
\qquad
{\text{and}}
\qquad
{\mathbb I}^*_k={\mathbb I}_{ks}{\mathbb I}_{kt}{\mathbb I}_{ku},
\end{displaymath}
imply $\PP(\Delta_1\cap {\cal K}_t)
=
\E  {\mathbb I}_{\Delta_1}{\mathbb I}_{{\cal K}_t}
\le
\E S_t{\tilde S}$. We show that $\E S_t{\tilde S}=O(t^{-3})$.
We split $S_t{\tilde S}={\tilde S}_1+{\tilde S}_2$, 
\begin{displaymath}
{\tilde S}_1
=
\sum_{i\in T}\sum_{j\in T\setminus \{i\}}
{\mathbb I}_{is}{\mathbb I}_{it}{\mathbb I}_{jt}{\mathbb I}_{ju}({\mathbb I}^*_i+{\mathbb I}^*_j),
\qquad
{\tilde S}_2
=
\sum_{(i,j,k)\in{\tilde{\mathbb Q}}}
{\mathbb I}_{is}{\mathbb I}_{it}{\mathbb I}_{jt}{\mathbb I}_{ju}{\mathbb I}^*_k,
\end{displaymath}
and estimate
\begin{eqnarray}
\nonumber
\E {\tilde S}_1
&
\le
&
\E \sum_{i\in T}\sum_{j\in T\setminus \{i\}}\lambda_{is}\lambda_{it}\lambda_{jt}\lambda_{ju}(\lambda_{iu}+\lambda_{js})
=O(t^{-3}),
\\
\label{KSQ3++}
\E {\tilde S}_2
&
\le
&
\E {\tilde S}'_2
\le
\E \sum_{(i,j,k)\in{\tilde{\mathbb Q}}}\lambda_{is}\lambda_{it}\lambda_{jt}\lambda_{ju}\lambda_{ks}\lambda_{ku}
=O(t^{-3}).
\end{eqnarray}
Here ${\tilde S}'_2$ is defined in the same way as ${\tilde S}_2$, but with ${\mathbb I}^*_k$ replaced by 
${\mathbb I}'_k={\mathbb I}_{ks}{\mathbb I}_{ku}$.

\end{proof}

\subsection{Proof of (\ref{s-momentai})}

We only give a sketch of the proof.
Let $s<t$ satisfy the inequality $\lceil at\rceil\le \lfloor bs\rfloor$. An 
attribute $w_i$ is called  witness of the edge $v_s\sim v_t$  whenever 
${\mathbb I}_{it}{\mathbb I}_{is}=1$.
The sums 
\begin{displaymath}\label{Assrt-0}
e_{st}=\sum_{i\in T_s\cap T_t}{\mathbb I}_{is}{\mathbb I}_{it}
\qquad
{\text{and}}
\qquad
q_{st}=\sum_{\{i,j\}\subset T_s\cap T_t}{\mathbb I}_{is}{\mathbb I}_{it} {\mathbb I}_{js}{\mathbb I}_{jt}
\end{displaymath}
count witnesses and pairs of witnesses of the edge $v_s\sim v_t$, respectively. 
We write, by inclusion-exclusion,
\begin{equation}
e_{st}-q_{st}
\le 
{\mathbb I}_{\{v_s\sim v_t\}}
\le 
e_{st}
\end{equation}
and note that the quadratic term $q_{st}$ is negligibly small. 
Hence, we approximate 
\begin{equation}\label{e-st+}
 {\mathbb I}_{\{v_s\sim v_t\}}= e_{st}(1+o_P(1)),
\qquad
\PP(v_s\sim v_t)=(1+o(1))\E e_{st}.
\end{equation}

Given $t$ and $i,j\in T_t$, we denote  $T_{it}^*=T_i^*\setminus\{t\}$ and introduce random variables
\begin{displaymath}
u_{it}=\sum_{k\in T_{it}^*}{\mathbb I}_{ik}, 
\qquad
z_{ijt}=\sum_{k\in T_{it}^*\cap T_{jt}^*}{\mathbb I}_{ik}{\mathbb I}_{jk},
\qquad
L_t=\sum_{i\in T_t}{\mathbb I}_{it}u_{it},
\qquad
Q_t=\sum_{\{i,j\}\subset T_t}{\mathbb I}_{it}{\mathbb I}_{jt}z_{ijt}.
\end{displaymath}
We remark that $L_t$ counts pairs $(v_s\sim v_t; w_i)$, where 
$w_i$ is a witness of the edge $v_s\sim v_t$ in
 $G_{X,Y}$, for some $v_s\in W\setminus\{v_t\}$. In particular,  we have $d(v_t)\le L_t$. 
Similarly, $Q_t$ counts all triples $(v_s\sim v_t; w_i, w_j)$, where $w_i$ and $w_j$ 
are distinct witnesses
of an edge $v_s\sim v_t$. Note that a neighbour $v_s$ of $v_t$, which has 
$k$ witnesses of the edge $\{v_s\sim v_t\}$, 
contributes $1$ to the number $d(v_t)$ of
neighbours of $v_t$. It contributes $k$ to the sum $L_t$ and it 
contributes $\tbinom{k}{2}$ to the sum $Q_t$. Hence, we always have
\begin{equation}\nonumber
L_t-Q_t\le d(v_t)\le L_t.
\end{equation}
We note that the quadratic term  $Q_t$ is negligibly small 
and approximate
$d(v_t)= L_t(1+o_P(1))$. Combining this approximation with
(\ref{e-st+}) we obtain, for $r=1,2$ and $u=s,t$,
\begin{equation}\label{12-22-1+}
 \E_{st}d^r(v_u)=(\E e_{st})^{-1}\E e_{st}L^r_u+o(1)
\quad
\
{\text{and}}
\quad
\
\E_{st}d(v_s)d(v_t)= (\E e_{st})^{-1}\E e_{st}L_sL_t+o(1).
\end{equation}

Next we evaluate  expectations in the right-hand sides of (\ref{12-22-1+}).
A straightforward but tedious calculation shows that
\begin{eqnarray}
\nonumber
%\label{2013-Est}
\E e_{st}
&=&
\Theta(1+o(1)) h_1,
\\
\nonumber
\E e_{st}L_s
&=&
\E e_{st}L_t+o(\Theta)=
\Theta(1+o(1))(h_1+h_2+2h_3),
\\
\nonumber
\E e_{st}L^2_s
&=&
\E  e_{st}L^2_t+o(\Theta)
=
\Theta(1+o(1))(h_1+3h_2+6h_3+h_4+6h_5+4h_6),
\\
\nonumber
\E e_{st}L_sL_t
&=&
\Theta(1+o(1))(h_1+3h_2+4h_3+h_4+4h_5+4h_7).
\end{eqnarray}
Here  we denote $\Theta=(st)^{-1/2}\ln(bs/at)$. We recall that $h_i$ are defined in 
(\ref{estLs}) above.
Now (\ref{s-momentai}) follows from  (\ref{12-22-1+}).

\bigskip

{\it Acknowledgement}. 
Research of M. Bloznelis
was supported  by the  Research Council of Lithuania grant MIP-053/2011.

%\end{document}

\newpage

\end{document}